\documentclass{amsart}%
\usepackage{amssymb}
\usepackage{amsmath}
\usepackage{amsfonts}

\usepackage[all,cmtip]{xy}
\usepackage[margin=1.6in]{geometry}%
\usepackage{graphicx}

\usepackage{hyperref}

\newtheorem{theorem}{Theorem}
\theoremstyle{plain}
\newtheorem*{acknowledgement}{Acknowledgement}

\newtheorem{corollary}[theorem]{Corollary}
\newtheorem*{corollary_nocount}{Corollary}

\newtheorem{definition}[theorem]{Definition}

\newtheorem{lemma}[theorem]{Lemma}

\newtheorem{proposition}[theorem]{Proposition}
\newtheorem{Fact}[theorem]{Fact}
\theoremstyle{remark}
\newtheorem{remark}[theorem]{Remark}
\newtheorem{example}[theorem]{Example}

\newtheorem{problem}{Problem}
\numberwithin{equation}{section}
\numberwithin{theorem}{section}
\begin{document}
\title[Multiplicative zeta function]{Multiplicative zeta function and logarithmic\linebreak point counting over
finite fields}
\author{O. Braunling}
\thanks{The author was supported by the DFG GK1821 \textquotedblleft Cohomological
Methods in Geometry\textquotedblright.}

\begin{abstract}
The zeta function of a motive over a finite field is multiplicative with
respect to the direct sum of motives. It has beautiful analytic properties, as
were predicted by the Weil conjectures. There is also a multiplicative zeta
function, which instead respects the tensor product of motives. There is no
analogue of the Weil conjectures, and we give a sufficient criterion for an
analytic continuation to exist. This happens, for example, for cellular
varieties, abelian varieties, or genus $\geq2$ curves with a supersingular Jacobian.
\end{abstract}
\maketitle

\section{Introduction}

Let $X/\mathbf{F}_{q}$ be a variety over a finite field. The usual zeta
function%
\begin{equation}
Z(X,t):=\exp\left(  \sum\nolimits_{r\geq1}\left\vert X(\mathbf{F}_{q^{r}%
})\right\vert \cdot\frac{t^{r}}{r}\right) \label{lm1}%
\end{equation}
behaves well under disjoint union, $Z(X_{1}\coprod X_{2},t)=Z(X_{1},t)\cdot
Z(X_{2},t)$. Generalized to motives, it respects the symmetric monoidal
structure coming from the direct sum of motives. However, we might instead be
interested in the question: Can our variety $X$ be written as a product
$X=X_{1}\times X_{2}\,$? For this type of question the multiplicative zeta
function%
\[
Z_{\log}(X,t):=\exp\left(  \sum\nolimits_{r\geq1}\log\left\vert X(\mathbf{F}%
_{q^{r}})\right\vert \cdot\frac{t^{r}}{r}\right)  \text{,}%
\]
when defined, is better suited. It satisfies%
\[
Z_{\log}(X_{1}\times X_{2},t)=Z_{\log}(X_{1},t)\cdot Z_{\log}(X_{2},t)\text{.}%
\]
The definition can also be generalized to motives, and then respects the
symmetric monoidal structure coming from the tensor product of motives. In a
way, $Z$ resp. $Z_{\log}$ belong to the two natural symmetric monoidal
structures on a Tannakian category, \textquotedblleft$\oplus$%
\textquotedblright\ resp. \textquotedblleft$\otimes$\textquotedblright%
.\medskip

We know a lot about the ordinary zeta function thanks to the Weil conjectures,
for example:\newline\textsc{(A)} The function $Z$ is rational; in particular
it has an analytic continuation to the entire complex plane.\newline%
\textsc{(B)} Poincar\'{e} Duality of a smooth projective variety $X$ induces a
functional equation%
\[
Z(X,(q^{d}t)^{-1})=\pm q^{d\chi(X)}\cdot t^{\chi(X)}\cdot Z(X,t)\text{.}%
\]
\textsc{(C)} Zeros and poles can be described in terms of the cohomology of
$X$.

And of course all properties of $Z$ follow from the existence of a Weil
cohomology theory. There is no indication that the function $Z_{\log}$ can be
obtained from something like a Grothendieck--Lefschetz trace formula, the
logarithm term is just too disruptive, so its analytic properties are far less
clear. Basically, following Murphy's Law, one might suspect its properties are
random at best.

But this is not so. Firstly, we shall show that $Z_{\log}$ behaves well for
all abelian varieties:

\begin{theorem}
\label{thm_i_1}Let $A/\mathbf{F}_{q}$ be an abelian variety of dimension
$g\geq1$. Then $Z_{\log}$ has radius of convergence $1$. If $\alpha_{1}%
,\ldots,\alpha_{2g}$ denote the Weil $q$-numbers of weight one, $Z_{\log}$
admits a multi-valued analytic continuation to%
\[
\mathbf{C}\setminus\{1,\alpha_{1}^{\mathbf{Z}_{\geq2}},\ldots,\alpha
_{2g}^{\mathbf{Z}_{\geq2}}\}\text{.}%
\]

\end{theorem}

See Theorem \ref{Thm_AC_ForAbelianVar} for a precise statement. We will give a
precise formulation for `multi-valued analytic continuation' below. Secondly,
we show that $Z_{\log}$ has an analytic continuation for all cellular
varieties. Indeed, it suffices if all summands in the motivic decomposition of
the variety are (Tate or) supersingular. The latter means that all its
Frobenius eigenvalues are of the shape $\zeta\cdot q^{w/2}$ for $\zeta$ a root
of unity and $w$ an integer. This encompasses all Artin and Tate motives. If
one believes in the\ Tate conjecture, numerical pure motives over
$\mathbf{F}_{q}$ are generated by all abelian varieties; and these
supersingular motivic summands would be those coming from the supersingular
abelian varieties.

\begin{theorem}
\label{thm_i_2}Suppose $X/\mathbf{F}_{q}$ is a smooth projective variety with
an $\mathbf{F}_{q}$-rational point. Suppose its motive\footnote{Here `motive'
refers to pure Grothendieck motives. It can be taken to mean numerical
motives, or $\ell$-adic homological motives for any $\ell$ different from the
characteristic of $\mathbf{F}_{q}$.} splits as a direct sum%
\[
\mathcal{M}(X)=\bigoplus M_{i}%
\]
such that each summand $M_{i}$ is supersingular, e.g. a Tate motive. Then
$Z_{\log}$ has a multi-valued analytic continuation to%
\[
\mathbf{C}\setminus\Delta\text{,}%
\]
with $\Delta$ some discrete subset of $\mathbf{C}$.
\end{theorem}

See Theorem \ref{Thm_AC_ForTateAndSupersingMotivicDecomp} for details. This
theorem covers for example: projective space, Grassmannians, or more broadly
all projective homogeneous varieties. It also covers smooth projective curves
of arbitrary genus, as long as their Jacobian is supersingular.

We also extend the definition of $Z_{\log}$ to motives. It cannot always be
defined then, but whenever it exists, it is multiplicative with respect to the
tensor product of motives. There are plenty of motives not coming from a
variety, for which we also get the existence of analytic continuations.

The above theorems also have an implication which no longer makes any
reference to $Z_{\log}$:

\begin{corollary_nocount}
If $X/\mathbf{F}_{q}$ is a smooth projective variety of dimension $\geq1$,
meeting the hypotheses of Theorem \ref{thm_i_1} or Theorem \ref{thm_i_2}, then
the sequence%
\[
n\mapsto\log\left\vert X(\mathbf{F}_{q^{n}})\right\vert
\]
does not satisfy any linear recurrence equation.
\end{corollary_nocount}

This result may not be particularly important; and probably admits a direct
proof based on the rationality of $Z$. However, it falls out with no extra
work from the previous results: If the sequence satisfies a linear recurrence,
then its generating function, which is nothing but $Z_{\log}^{\prime}/Z_{\log
}$, would be rational. Rational functions have a single-valued analytic
continuation to $\mathbf{C}\setminus\{$finite set$\}$, contradicting the
analytic properties which our theory yields. This needs the more precise
versions in the main body of the text, and not the shortened formulations
above. See Theorem \ref{thm_LinRecur}.

\begin{corollary_nocount}
Suppose $X/\mathbf{F}_{q}$ is a geometrically connected smooth projective
curve with an $\mathbf{F}_{q}$-rational point. If

\begin{enumerate}
\item the genus is $g=0,1$ or

\item the genus is $g\geq2$ and the Jacobian of $X$ is supersingular,
\end{enumerate}

then $Z_{\log}(X,t)$ admits a multi-valued analytic continuation to
$\mathbf{C}\setminus\Delta$, with $\Delta$ some discrete subset of
$\mathbf{C}$.
\end{corollary_nocount}

See Theorem \ref{thm_caseofcurves}

\begin{acknowledgement}
I thank Giuseppe Ancona for his help and very clarifying explanations
regarding a number of questions. I thank Fritz H\"{o}rmann for his help around Honda--Tate.
\end{acknowledgement}

\section{Definitions}

\subsection{Conventions\label{subsect_Conventions}}

For us, a \emph{variety} $X/k$ is a finite type separated $k$-scheme for $k$
some field. A morphism of varieties is tacitly understood to mean a finite
type separated $k$-morphism. Suppose $k=\mathbf{F}_{q}$ is a finite field. By
\textquotedblleft\emph{Frobenius}\textquotedblright\ we always refer to the
geometric Frobenius, i.e. it acts as $x\mapsto x^{q^{-1}}$ on elements
$x\in\mathbf{F}_{q}$. This agreement only really plays a r\^{o}le to ensure
that $\mathbf{Z}_{\ell}(1):=\underleftarrow{\lim}\mu_{\ell^{n}}$ has weight
$-2$.

The term \textquotedblleft\emph{logarithm}\textquotedblright\ will usually
refer to the standard branch, i.e.%
\[
\log z=\log\left\vert z\right\vert +i\arg z\qquad\text{with}\qquad\arg
z\in(-\pi,\pi]
\]
for all $z\in\mathbf{C}^{\times}$. When we work with a more general branch of
the logarithm, we denote it by a capitalized `$\operatorname*{Log}$'.

We will freely use some aspects of the theory of pure (Grothendieck) motives.
All we need is explained in \cite{MR1265538} or \cite{MR2115000}. Our
conventions are as follows: Let $F$ be any field of characteristic zero, which
will serve as our field of coefficients. We pick $F$ once and for all and
henceforth drop it from the notation. Let $\mathsf{Mot}_{\sim}(k)$ be the
category of effective pure (Grothendieck) motives over $k$ with coefficients
in $F$, and \textquotedblleft$\sim$\textquotedblright\ denotes an adequate
equivalence relation. Objects are of the shape $(X,p)$ with $X$ a smooth
projective $k$-variety and $p$ an idempotent correspondence from $X$ to itself.

Concretely, we write $\mathsf{Mot}_{num}(k)$ for numerical motives, and
$\mathsf{Mot}_{hom_{\ell}}(k)$ for homological motives ($\ell\neq
\operatorname*{char}k$), using $\ell$-adic cohomology as the underlying Weil
cohomology theory. Conjecturally, homological and numerical equivalence agree,
and in particular the choice of a Weil cohomology theory should not matter.
So, speculatively, $\mathsf{Mot}_{num}(k)=\mathsf{Mot}_{hom_{\ell}}(k)$.
However, this remains open. Many aspects of the formalism in this text can be
extended to motives. We discuss this in the Appendix
\S \ref{sect_MotivicPicture}.

\subsection{Definition}

First of all, we give a definition of $Z_{\log}$ for varieties.

\begin{definition}
\label{def_Zlog_ForVars}If $X/\mathbf{F}_{q}$ is a variety with an
$\mathbf{F}_{q}$-rational point, we define the \emph{multiplicative zeta
function} as the power series%
\[
Z_{\log}(X,t):=\exp\left(  \sum\nolimits_{r\geq1}\log\left\vert X(\mathbf{F}%
_{q^{r}})\right\vert \cdot\frac{t^{r}}{r}\right)  \text{.}%
\]

\end{definition}

Thanks to the condition $X(\mathbf{F}_{q})\neq\varnothing$, we have
$\left\vert X(\mathbf{F}_{q^{r}})\right\vert \geq1$ for all $r\geq1$, making
this expression well-defined as a formal power series over the reals. If
$X_{1},X_{2}$ are varieties with $X_{i}(\mathbf{F}_{q})\neq\varnothing$ for
$i=1,2$, we get the fundamental property:%
\[
Z_{\log}(X_{1}\times X_{2},t)=Z_{\log}(X_{1},t)\cdot Z_{\log}(X_{2},t)\text{.}%
\]
For the ordinary zeta function, denoted by $Z$, we instead have%
\[
Z(X_{1}%
{\textstyle\coprod}
X_{2},t)=Z(X_{1},t)\cdot Z(X_{2},t)\text{.}%
\]

\begin{remark}
There is also a formula for product varieties for the function $Z$, but it
relies on the more complicated so-called \emph{Witt product} `$\ast$' (it is
not due to Witt, but named so as it is related to the ring of big Witt
vectors). Then $Z(X_{1}\times X_{2},t)=Z(X_{1},t)\ast Z(X_{2},t)$. We refer to
\cite{MR2272145} and \cite{MR3395874} for more on this perspective.
\end{remark}

\begin{definition}
\label{def_AC}We say that a series $f(t)=\sum b_{r}t^{r}$, $b_{r}\in
\mathbf{C}$, with positive radius of convergence, \emph{has a (possibly
multi-valued) analytic continuation}, or in brief $\left.
\text{\textbf{$\textsc{(AC)}$}}\right.  $, if there exists a discrete set
$\Delta\subset\mathbf{C}$, $0\notin\Delta$, with the property: For every
simply connected domain $U$ with $0\in U\subset\mathbf{C}\setminus\Delta$,
there exists a holomorphic function $f_{U}:U\rightarrow\mathbf{C}$ such that
$f_{U}$ agrees with $f$ in some neighbourhood of $t=0$.
\end{definition}

Equivalently, regard $(\mathbf{C}\setminus\Delta,0)$ as a pointed space. If
$V$ denotes a sufficiently small neighbourhood of $0$, the above datum defines
a unique lift of $f$ on $V$ to the universal covering space $(\tilde{X},\ast
)$:%
\begin{equation}%
%
%
\xymatrix{
& (\tilde{X},\ast) \ar[d] \\
(V,0) \ar[r] \ar[ur] & (\mathbf{C} \setminus\Delta, 0),
}%
\label{lpic1}%
\end{equation}
Giving this lift of $f$ is equivalent to providing the collection of all
$(f_{U})_{U}$ as above. This formulation is more elegant, but less practical
for explicit computations.

We will take Definition \ref{def_AC} as the meaning for the term `analytic
continuation' in this text in order to avoid having to repeat more precise
qualifiers again and again. We call it multi-valued because different choices
of $U$ might yield different continuations.

\begin{remark}
[Existence]The existence of a (single- or multi-valued) analytic continuation
is a non-trivial statement. Even for the ordinary zeta function, as in
Equation \ref{lm1}, having an explicit power series expansion does not easily
let us read off whether $Z$ is a rational function. For example, the power
series%
\[
\text{\textsf{(A)}}\quad\text{ }\sum_{r\geq1}t^{r},\qquad\text{\textsf{(B)}%
}\quad\text{ }\sum_{r\geq1}\frac{t^{r}}{r},\qquad\text{\textsf{(C)}}%
\quad\text{ }\sum_{r\geq1}t^{2^{r}}\text{,}%
\]
all have radius of convergence precisely one. The first one is a rational
function, namely $\frac{t}{1-t}$, and thus admits a meromorphic continuation
to the entire plane, while the second is $-\log(1-t)$, so while it does admit
a holomorphic extension to all of $\mathbf{C}\setminus\{0\}$, it requires
multiple branches, and yet the last power series has the unit circle as its
natural boundary. This means that it is impossible to find an analytic
continuation anywhere outside the open unit disc -- for a dense set inside the
unit circle, its values tend to go off to infinity as one approaches the
radius of convergence. Both \textsf{(A)}\ and \textsf{(B)} satisfy our
definition of \textbf{$\left.  \text{\textbf{$\textsc{(AC)}$}}\right.  $},
while \textsf{(C)} does not.
\end{remark}

\begin{example}
\label{example_AffineSpace}For affine space we have $\left\vert \mathbf{A}%
^{n}(\mathbf{F}_{q^{r}})\right\vert =q^{nr}$ and thus%
\[
Z_{\log}(\mathbf{A}^{n},t)=\exp\left(  \sum_{r\geq1}\log(q^{nr})\cdot
\frac{t^{r}}{r}\right)  =\exp\left(  n\log(q)\cdot\frac{t}{1-t}\right)
=\left(  q^{\left(  \frac{t}{1-t}\right)  }\right)  ^{n}\text{.}%
\]
This defines a single-valued holomorphic continuation to all of $\mathbf{C}%
\setminus\{1\}$. We have \textbf{$\left.  \text{\textbf{$\textsc{(AC)}$}%
}\right.  $} for $\Delta:=\{1\}$. We also see the property of the
multiplicativity; it would have sufficed to deal with $\mathbf{A}^{1}$. (The
usual zeta function is $Z(\mathbf{A}^{n},t)=\frac{1}{1-q^{n}t}$)
\end{example}

\begin{example}
Suppose we want to deal with the torus $(\mathbf{G}_{m})^{n}$ resp.
$(\mathbf{P}^{1})^{\times n}$. It suffices to treat $n=1$. However, we get%
\[
Z_{\log}(\mathbf{-},t):=\exp\left(  \sum\nolimits_{r\geq1}\log(q^{r}\mp
1)\cdot\frac{t^{r}}{r}\right)  \text{.}%
\]
The radius of convergence of the inner series is $R=1$, and the values of
$\log(q^{r}\mp1)$ will always be very close to $r\log q$, yet not quite the
same. So the question whether \textbf{$\left.  \text{\textbf{$\textsc{(AC)}$}%
}\right.  $} holds is a priori unclear. Later, we will be able to answer this affirmatively.
\end{example}

See the Appendix, \S \ref{sect_MotivicPicture}, for the extension of the
definition of $Z_{\log}$ to motives. Most of this text can be read without
having to deal with motives.

\subsection{Pseudo-divisors}

We shall use the word `divisor' in the sense of complex manifolds, i.e.
instead of defining it to be a finite linear combination as customary in
algebraic geometry, we just demand local finiteness in the complex topology:

\begin{definition}
\label{def_PseudoDivisor}A\ \emph{pseudo-divisor} on $\mathbf{C}$ is a
set-theoretic function%
\[
\mathcal{D}:\mathbf{C}\longrightarrow\mathbf{Z}\cup\{\infty\}\text{.}%
\]
We may express this datum in the notation $\mathcal{D}=\sum_{P\in\mathbf{C}%
}n_{P}[P]$ with $n_{P}\in\mathbf{Z}\cup\{\infty\}$, reminiscent of divisors.
Define the \emph{support} of $\mathcal{D}$ by%
\[
\operatorname*{supp}\mathcal{D}:=\overline{\{P\in\mathbf{C}\mid n_{P}\neq
0\}}\text{,}%
\]
where the closure is taken with respect to the complex topology (not
Zariski!). We say that $\mathcal{D}$ is a \emph{locally finite divisor} on
$\mathbf{C}$ if the support of $\mathcal{D}$ is locally finite, i.e. for any
point $z\in\mathbf{C}$ there exists an open neighbourhood of $z$ which
contains only finitely many points in the support of $\mathcal{D}$.
\end{definition}

\begin{definition}
\label{def_PeriodicDivisor}Suppose $\mathcal{D}$ is a pseudo-divisor. We also
define a $2\pi i$-periodic version, called $\mathcal{D}^{\operatorname*{per}%
,\pm}$, of a pseudo-divisor by%
\[
\mathcal{D}^{\operatorname*{per},\pm}:=\sum_{j=1}^{\infty}T_{j}^{\ast
}\mathcal{D}\text{,}%
\]
where $T_{j}$ is the translation $z\mapsto z\pm2\pi ij$ (i.e. $T_{j}^{\ast} $
translates the divisor $\mathcal{P}$ by the multiple $2\pi ij$ in the
plane).\ We write $\mathcal{D}^{\operatorname*{per}}$ if $j$ runs through all
of $\mathbf{Z}$, so $\mathcal{D}^{\operatorname*{per}}:=\sum_{j=-\infty
}^{\infty}T_{j}^{\ast}\mathcal{D}$. If for any point $P\in\mathbf{C}$ these
definitions would require us to evaluate a sum of infinitely many non-zero
terms, we define the multiplicity of the sum at $P$ to be $n_{P}=\infty$.
\end{definition}

\section{Step I}

In this section we will begin developing the technical tools necessary to
establish the existence of an analytic continuation.

\subsection{Idea}

Our method is as follows: We want to apply Abel--Plana summation, which is a
technique which succeeds excellently in producing analytic continuations for
the polylogarithm $\operatorname*{Li}\nolimits_{s}(z)$ or the Hurwitz zeta
function $\mathbf{\zeta}(s;q)$. It belongs to the family of results around
Euler--MacLaurin summation. In brief: Firstly, we transform the power series
in question into an integral, and then we evaluate the integral in a different
fashion. If this is designed in a careful way, one may arrange to arrive at an
expression which remains sensible in a larger domain of definition than the
original power series. As we shall see, this strategy frequently succeeds for
$Z_{\log}$. Whether it does, will turn out to be controlled by a certain pseudo-divisor.

References are Olver's book \cite[Ch. 8, \S 3]{MR1429619} or Hardy's classic
treatise \cite[\S 13.14]{MR1188874}. The statement is as follows:

\begin{proposition}
[Abel, Plana]\label{Prop_AbelPlanaFormula}Suppose $h:\{s\mid\operatorname*{Re}%
s\geq0\}\rightarrow\mathbf{C}$ is a function such that the following
conditions are met:

\begin{enumerate}
\item For every integer $n\geq0$ and the vertical strip $\mathsf{S}%
_{n}:=\{s\mid n\leq\operatorname*{Re}s\leq n+1\}$, the function $h$ is
continuous in $\mathsf{S}_{n}$ and is holomorphic in the interior of
$\mathsf{S}_{n}$ and at $s=0$.

\item For every $a\geq0$, we have $\underset{b\rightarrow\pm\infty}{\lim
}\left\vert h(a+bi)\right\vert \cdot e^{-2\pi\left\vert b\right\vert }=0$ and
this convergence is uniform in $a$.

\item We have $\underset{a\rightarrow+\infty}{\lim}\int_{0}^{\infty}\left\vert
h(a\pm bi)\right\vert e^{-2\pi b}\mathrm{d}b=0$.
\end{enumerate}

Then the identity%
\[
\sum_{n=0}^{\infty}h(n)=\int_{0}^{\infty}h(s)\mathrm{d}s+\frac{1}{2}%
h(0)+i\int_{0}^{\infty}\frac{h(ib)-h(-ib)}{e^{2\pi b}-1}\mathrm{d}b
\]
holds and the integrals on the right-hand side exist.
\end{proposition}

We shall shortly see that we will need to modify this method a little bit in
order to apply it to our problem.

\subsection{Setup\label{subsect_Setup}}

In this section we shall work with some running assumptions and notation
throughout: Pick $N\geq1$. Suppose for $i=1,\ldots,N$ we are given complex
numbers $\varepsilon_{i},\lambda_{i}\in\mathbf{C}$ such that $0<\left\vert
\lambda_{i}\right\vert <1$. Define a pseudo-divisor (in the sense of
Definition \ref{def_PseudoDivisor}) and depending on our data $(\varepsilon
_{i},\lambda_{i})_{i=1,\ldots,N}$ by%
\begin{equation}
\mathcal{P}:=\sum_{l=1}^{\infty}\sum_{k_{1}+\cdots+k_{N}=l}\binom{l-1}%
{k_{1},\ldots,k_{N}}\varepsilon_{1}^{k_{1}}\cdots\varepsilon_{N}^{k_{N}}%
\cdot\left[  k_{1}\log\lambda_{1}+\cdots+k_{N}\log\lambda_{N}\right]
\text{.}\label{leq_Def_P}%
\end{equation}
Or, equivalently, one might prefer: To any point $z\in\mathbf{C}$ attach a
multiplicity by%
\[
\sum_{l=1}^{\infty}\sum_{k_{1}+\cdots+k_{N}=l}\binom{l-1}{k_{1},\ldots,k_{N}%
}\varepsilon_{1}^{k_{1}}\cdots\varepsilon_{N}^{k_{N}}\cdot\delta_{z=\sum
k_{i}\log\lambda_{i}}\text{,}%
\]
where $\delta_{A}=1$ if $A$ is a true statement, and $\delta_{A}=0$ otherwise.
If there are infinitely many such summands, we just declare the multiplicity
to be $\infty$.

We define $M:=\max_{i}\{\left\vert \varepsilon_{i}\right\vert \}$. If
$r_{0}\geq1$ is any integer, we will intend to expand the logarithm in the
expression%
\[
\sum_{r\geq r_{0}}\log\left(  1-\sum_{i=1}^{N}\varepsilon_{i}\lambda_{i}%
^{r}\right)  \cdot e^{-wr}%
\]
in terms of its usual power series $\log(1-z)=-\sum_{l\geq1}l^{-1}z^{l}$.
Ideally, we would like to pick $r_{0}:=1$. However, this is not necessarily
possible because for small $r$ the expression $\sum_{i=1}^{N}\varepsilon
_{i}\lambda_{i}^{r}$ need not lie within the radius of convergence of the
logarithm series. We fix this as follows:

We compute%
\[
\left\vert \sum_{i=1}^{N}\varepsilon_{i}\lambda_{i}^{r}\right\vert \leq
M\sum_{i=1}^{N}\left\vert e^{r\log\lambda_{i}}\right\vert \leq M\cdot\max
_{i}\{\left\vert \lambda_{i}\right\vert \}^{\operatorname{Re}(r)}\cdot
\sum_{i=1}^{N}e^{-\operatorname{Im}(r)\arg(\lambda_{i})}\text{,}%
\]
where the second inequality stems from the computation $\left\vert
e^{r\log\lambda_{i}}\right\vert =\left\vert e^{x\log\left\vert \lambda
_{i}\right\vert -y\arg\lambda_{i}}\right\vert =\left\vert \lambda
_{i}\right\vert ^{\operatorname{Re}r}e^{-\operatorname{Im}(r)\arg(\lambda
_{i})}$, where we have written $r=x+iy$ with $x,y\in\mathbf{R}$. Thus, if we
pick some $K>0$ and constrain%
\[
\left\vert \operatorname{Im}(r)\right\vert \leq K\text{,}%
\]
then there exists some sufficiently large integer $r_{0}\geq1$ such that the
following two properties hold:

\begin{enumerate}
\item For all $i=1,\ldots,N$, and all $r$ with $\operatorname{Re}(r)\geq
r_{0}$ and $\left\vert \operatorname{Im}(r)\right\vert \leq K$,%
\begin{equation}
\left\vert e^{r\log\lambda_{i}}\right\vert <\frac{1}{2}\text{.}\label{la0}%
\end{equation}

\item And moreover, for all $r$ with $\operatorname{Re}(r)\geq r_{0}$ and
$\left\vert \operatorname{Im}(r)\right\vert \leq K$,%
\begin{equation}
\left\vert \sum_{i=1}^{N}\varepsilon_{i}\lambda_{i}^{r}\right\vert <\frac
{1}{2}\text{.}\label{la1}%
\end{equation}

\end{enumerate}

It is here where we have used the condition $\left\vert \lambda_{i}\right\vert
<1$. From now on, fix once and for all some $K>0$ and pick $r_{0}$ (tacitly
depending on $\varepsilon_{i},\lambda_{i},K$) so that both inequalities,
\ref{la0}, \ref{la1} hold.

\begin{remark}
\label{rmk_DecreaseK}Suppose we pick a different $K^{\prime}$ such that
$0<K^{\prime}<K$. Then the above conditions still remain valid for the same
choice of $r_{0}$.
\end{remark}

\subsection{Continuation of auxiliary functions}

With this data chosen, we shall show:

\begin{theorem}
\label{thm_BaseAnalyticContinuation}Suppose we are given $(\varepsilon
_{i},\lambda_{i})_{i=1,\ldots,N}$ and have picked $K,r_{0}$ as explained
above. Then for all $w\in\mathbf{C}$ in the open right half-plane we have the
equality%
\begin{align*}
\sum_{r=r_{0}}^{\infty}\log\left(  1-\sum_{i=1}^{N}\varepsilon_{i}\lambda
_{i}^{r}\right)  e^{-wr}  & =\frac{1}{2}\log\left(  1-\sum_{i=1}%
^{N}\varepsilon_{i}\lambda_{i}^{r_{0}}\right)  \cdot e^{-wr_{0}}\\
& +\sum_{l=1}^{\infty}\sum_{j\in\mathbf{Z}}\sum_{k_{1}+\cdots+k_{N}=l}%
\binom{l-1}{k_{1},\ldots,k_{N}}\varepsilon_{1}^{k_{1}}\cdots\varepsilon
_{N}^{k_{N}}\\
& \qquad\qquad\frac{e^{r_{0}\left(  k_{1}\log\lambda_{1}+\cdots+k_{N}%
\log\lambda_{N}-w+2\pi ij\right)  }}{k_{1}\log\lambda_{1}+\cdots+k_{N}%
\log\lambda_{N}-w+2\pi ij}\text{.}%
\end{align*}

\begin{enumerate}
\item The series on the left side is uniformly convergent in any compactum in
the open right half-plane.

\item The series on the right side is uniformly convergent in any compactum
$A\subset\mathbf{C}$ outside the support of $\mathcal{P}^{\operatorname*{per}%
}$, the periodification of the pseudo-divisor in Equation \ref{leq_Def_P}. The
support of $\mathcal{P}^{\operatorname*{per}}$ lies in the closed left half-plane.
\end{enumerate}
\end{theorem}

Note that the support of $\mathcal{P}^{\operatorname*{per}}$ might be the
entire closed left half-plane, so the statement of (2) can happen to be no
stronger than (1).\medskip

The rest of the section will be devoted to the proof of this. Define%
\begin{align}
& h:[r_{0},+\infty)\times i[-K,K]\longrightarrow\mathbf{C}\text{.}%
\label{l_def_h}\\
& h(r):=\log\left(  1-\sum_{i=1}^{N}\varepsilon_{i}\lambda_{i}^{r}\right)
\cdot e^{-wr}\text{,}\label{l_def_h2}%
\end{align}
where $\lambda^{r}$ means $\lambda^{r}:=\exp(r\cdot\log\lambda)$. Inside the
half-infinite box $[r_{0},+\infty)\times i[-K,K]$, we remain inside the radius
of convergence of the logarithm series, and in particular cannot hit the
branch cut of the outer logarithm. Thus, $h$ is holomorphic (and as is easy to
see, it is even holomorphic in an open neighbourhood of this box).

We begin with the following observation:

\begin{remark}
\label{rmk_estimate_integrand}Clearly $\log(1-z)$ is bounded inside any closed
disc inside the open unit disc, say $\leq C_{0}$. Then since line \ref{la1}
implies that $\sum_{i=1}^{N}\varepsilon_{i}\lambda_{i}^{r}$ lies inside this
circle, we get%
\[
\left\vert \log\left(  1-\sum_{i=1}^{N}\varepsilon_{i}\lambda_{i}^{r}\right)
e^{-wr}\right\vert \leq C_{0}\left\vert e^{-wr}\right\vert =C_{0}\left\vert
e^{-(\operatorname{Re}w)(\operatorname{Re}r)+(\operatorname{Im}%
w)(\operatorname{Im}r)}\right\vert \text{.}%
\]
Hence, given any compactum $A\subset\mathbf{C}$ in the open right half-plane,
then for any $w\in A$ and any sequence of values $r_{n}$ with $\left\vert
\operatorname{Im}r_{n}\right\vert $ bounded (e.g. $\leq K$) and
$\operatorname{Re}(r_{n})\rightarrow+\infty$, we have exponential decay of%
\[
\left\vert h(r)\right\vert =\left\vert \log\left(  1-\sum_{i=1}^{N}%
\varepsilon_{i}\lambda_{i}^{r}\right)  e^{-wr}\right\vert
\]
towards zero.
\end{remark}

So the initial idea would be to apply Proposition \ref{Prop_AbelPlanaFormula}
to the function in Equation \ref{l_def_h2}. However, this does not quite work
because we only have a function $h:[r_{0},+\infty)\times i[-K,K]\rightarrow
\mathbf{C}$, i.e. defined on a much smaller domain as would be required. Of
course the formula in Equation \ref{l_def_h2} can be extended to make sense on
all of $\mathbf{C}$, however \textsl{not in a holomorphic way}. Indeed, the
branch cut of the logarithm will generally (depending on $\varepsilon
_{i},\lambda_{i}$) make it impossible to meet the holomorphicity condition of
Proposition \ref{Prop_AbelPlanaFormula}.

\begin{example}
The following figure shows an example of the set of those $r\in\mathbf{C}$,
where $\operatorname{Re}(1-\lambda_{1}^{r}+\lambda_{2}^{r})\leq0$ for suitable
$\lambda_{1},\lambda_{2}$ on the left, and an example with $N=3$ on the right.
As one can see, the resulting geometry is complicated.
\[%
{\includegraphics[
height=1.5359in,
width=1.3794in
]%
{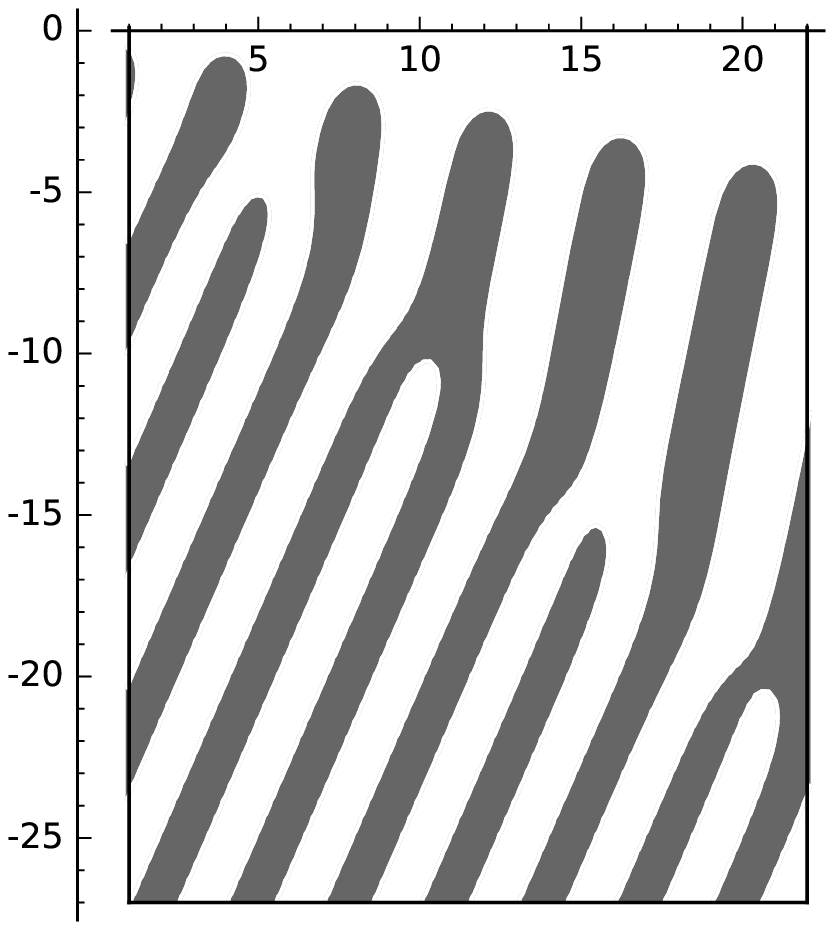}%
}
\qquad%
{\includegraphics[
height=1.5454in,
width=1.388in
]%
{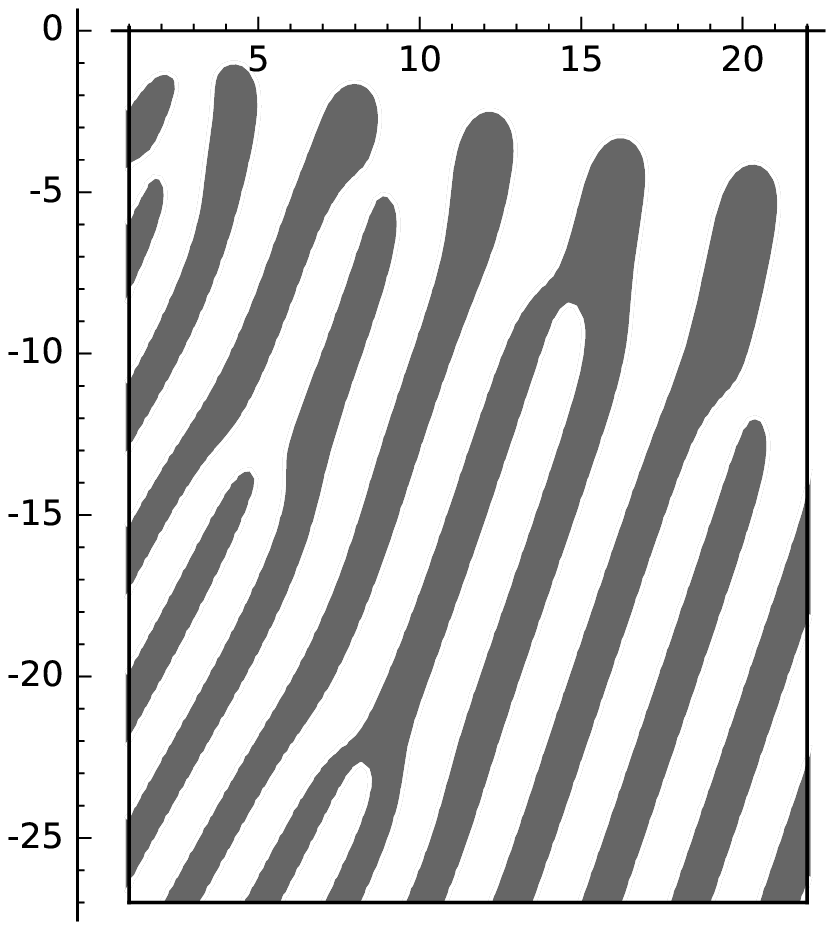}%
}
\]
The terms $\lambda_{i}^{r}=\exp(r\log\lambda_{i})$ have a periodicity built
in, caused by the $2\pi i$-periodicity of the exponential function. This
explains why we have so many distinct connected components. Correspondingly,
we get many pairwise disjoint copies of the logarithm branch cut, which lie
inside these copies of the negative half-plane.
\end{example}

\begin{example}
As a complex plot, the function $r\mapsto\log(1-\lambda_{1}^{r}+\lambda
_{2}^{r})$ can, depending on $\lambda_{1},\lambda_{2}$, for example look as
follows:%
\[%
{\includegraphics[
natheight=3.657700in,
natwidth=4.716900in,
height=2.0821in,
width=2.6826in
]%
{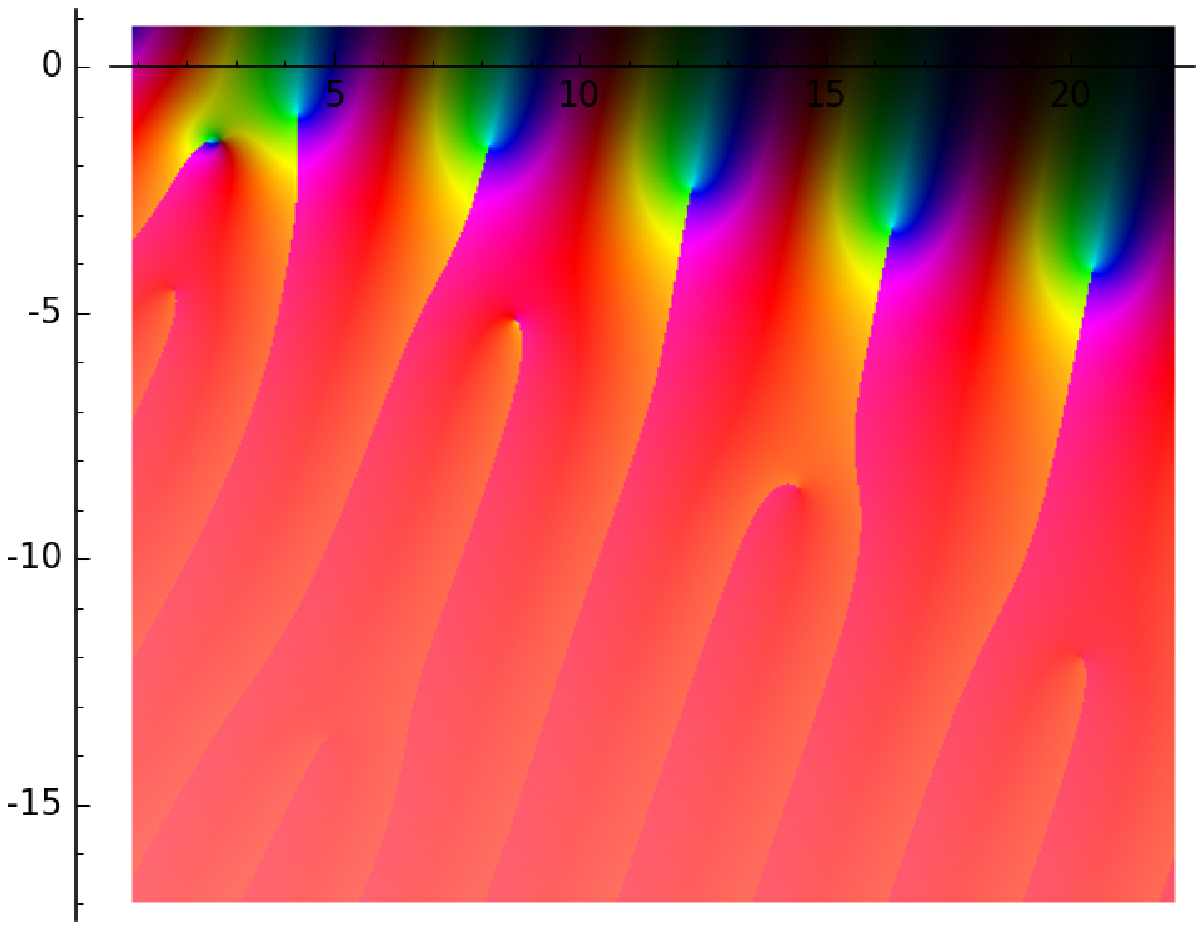}%
}
\]
One can see the jumps at the many copies of the logarithmic branch cut. Any
contour integration running through this territory is necessarily problematic.
It is best to avoid such regions altogether.
\end{example}

We shall therefore work with a more complicated variant of the Abel--Plana
method. Suppose $a\leq b$ are integers with $r_{0}\leq a$:

\begin{proposition}
\label{prop_TinyBoxAbelPlana}Suppose $h:[a,+\infty)\times i[-K,K]\rightarrow
\mathbf{C}$ is any function which admits a holomorphic continuation to an open
neighbourhood of this box. Then for all integers $1\leq a<b$, we have
\begin{align*}
\sum_{r=a}^{b}h(r)  & =\frac{1}{2}h(a)+\frac{1}{2}h(b)+\int_{a}^{b}%
h(s)\mathrm{d}s\\
& +i\int_{0}^{K}\frac{h(a+iy)-h(a-iy)}{e^{2\pi y}-1}\mathrm{d}y-i\int_{0}%
^{K}\frac{h(b+iy)-h(b-iy)}{e^{2\pi y}-1}\mathrm{d}y\\
& -\int_{a+iK}^{b+iK}\frac{h(s)}{1-e^{-2\pi is}}\mathrm{d}s+\int_{a-iK}%
^{b-iK}\frac{h(s)}{e^{2\pi is}-1}\mathrm{d}s\text{.}%
\end{align*}

\end{proposition}

\begin{proof}
This is proven in a similar fashion as the original result, see for example
\cite[Ch. 8, \S 3]{MR1429619}. A detailed proof is given as \cite[Prop.
4.3]{torsionhomology}.
\end{proof}

Next, using Prop. \ref{prop_TinyBoxAbelPlana} with the holomorphic function
of\ Equation \ref{l_def_h}, we obtain%
\begin{align}
& \sum_{r=a}^{b}\log\left(  1-\sum_{i=1}^{N}\varepsilon_{i}\lambda_{i}%
^{r}\right)  e^{-wr}=\frac{1}{2}\log\left(  1-\sum_{i=1}^{N}\varepsilon
_{i}\lambda_{i}^{a}\right)  \cdot e^{-wa}\nonumber\\
& \qquad\qquad\qquad+\frac{1}{2}\log\left(  1-\sum_{i=1}^{N}\varepsilon
_{i}\lambda_{i}^{b}\right)  \cdot e^{-wb}+\int_{a}^{b}h(s)\mathrm{d}%
s\label{l_ModifiedAbelPlana}\\
& \qquad\qquad\qquad+i\int_{0}^{K}\frac{h(a+iy)-h(a-iy)}{e^{2\pi y}%
-1}\mathrm{d}y-i\int_{0}^{K}\frac{h(b+iy)-h(b-iy)}{e^{2\pi y}-1}%
\mathrm{d}y\nonumber\\
& \qquad\qquad\qquad-\int_{a+iK}^{b+iK}\frac{h(s)}{1-e^{-2\pi is}}%
\mathrm{d}s+\int_{a-iK}^{b-iK}\frac{h(s)}{e^{2\pi is}-1}\mathrm{d}%
s\text{.}\nonumber
\end{align}
It remains to compute these integrals. We will first treat them for the
summation from $a$ to $b$ on the left-hand side, and then afterwards let
$b\rightarrow+\infty$. We begin with the integrals which appear in the last
line of the above equation:

\begin{proposition}
\label{prop_Vplusminuscases}Pick a sign \textquotedblleft$+/-$%
\textquotedblright. Suppose the pseudo-divisor $\mathcal{P}%
^{\operatorname*{per},\pm}$ (see Definition \ref{def_PeriodicDivisor}) is
locally finite. Then the integral%
\[
V^{+}(w):=\int_{a+iK}^{b+iK}\frac{\log\left(  1-\sum_{i=1}^{N}\varepsilon
_{i}\lambda_{i}^{r}\right)  \cdot e^{-wr}}{1-e^{-2\pi ir}}\mathrm{d}r\text{,}%
\]
resp.%
\[
V^{-}(w):=\int_{a-iK}^{b-iK}\frac{\log\left(  1-\sum_{i=1}^{N}\varepsilon
_{i}\lambda_{i}^{r}\right)  \cdot e^{-wr}}{e^{2\pi ir}-1}\mathrm{d}r\text{,}%
\]
exists and defines a holomorphic function on the open right half-plane. A
meromorphic analytic continuation to the entire complex plane is given by%
\begin{align*}
\tilde{V}^{\pm}(w)  & =\pm\sum_{l=1}^{\infty}\sum_{k_{1}+\cdots+k_{N}=l}%
\binom{l-1}{k_{1},\ldots,k_{N}}\varepsilon_{1}^{k_{1}}\cdots\varepsilon
_{N}^{k_{N}}\\
& \qquad\qquad\sum_{j=1}^{\infty}\left.  \frac{e^{r(k_{1}\log\lambda
_{1}+\cdots+k_{N}\log\lambda_{N}-w\pm2\pi ij)}}{k_{1}\log\lambda_{1}%
+\cdots+k_{N}\log\lambda_{N}-w\pm2\pi ij}\right\vert _{r=a\pm iK}^{r=b\pm iK}%
\end{align*}
All its poles have order $1$ and lie precisely at the support of
$\mathcal{P}^{\operatorname*{per},\pm}$. This series converges uniformly in
any compactum in $\mathbf{C}$ which avoids the support of $\mathcal{P}%
^{\operatorname*{per},\pm}$.
\end{proposition}

\begin{proof}
Note that for $r=x\pm iy$ with $x,y\in\mathbf{R}$ and $y>0$, we get
$\left\vert e^{2\pi i(x\pm iy)}\right\vert =e^{\mp2\pi y}$. Rewrite $V^{+}$ as%
\[
\int_{a+iK}^{b+iK}\frac{\log\left(  1-\sum_{i=1}^{N}\varepsilon_{i}\lambda
_{i}^{r}\right)  \cdot e^{-wr}}{1-e^{-2\pi ir}}\mathrm{d}r=\int_{a+iK}%
^{b+iK}\frac{\log\left(  1-\sum_{i=1}^{N}\varepsilon_{i}\lambda_{i}%
^{r}\right)  \cdot e^{-wr}}{-e^{-2\pi ir}(1-e^{2\pi ir})}\mathrm{d}r
\]
and since all throughout the path of integration we have $\operatorname{Im}%
r>0$, we have $\left\vert e^{2\pi ir}\right\vert <1$. Thus, we may unravel the
term $1-e^{2\pi ir}$ in terms of a geometric series, yielding%
\[
=-\sum_{j=1}^{\infty}\int_{a+iK}^{b+iK}\log\left(  1-\sum_{i=1}^{N}%
\varepsilon_{i}\lambda_{i}^{r}\right)  \cdot e^{(-w+2\pi ij)r}\mathrm{d}%
r\text{.}%
\]
For $V^{-}$ one proceeds analogously (in this case we will expand $1-e^{-2\pi
ir}$ as a geometric series. It converges since $r$ is in the open lower
half-plane and then $\left\vert e^{-2\pi ir}\right\vert <1$). Thus, to handle
both cases, it remains to treat%
\[
\int_{a\pm iK}^{b\pm iK}\log\left(  1-\sum_{i=1}^{N}\varepsilon_{i}\lambda
_{i}^{r}\right)  \cdot e^{(-w\pm2\pi ij)r}\mathrm{d}r\text{.}%
\]
As $r\in\lbrack a,b]\times\{\pm iK\}$, we have $\left\vert \operatorname{Im}%
r\right\vert \leq K$ and so we are inside the radius of convergence of the
logarithm series. We get%
\begin{align*}
& =-\sum_{l=1}^{\infty}\frac{1}{l}\int_{a\pm iK}^{b\pm iK}\left(  \sum
_{i=1}^{N}\varepsilon_{i}\lambda_{i}^{r}\right)  ^{l}\cdot e^{(-w\pm2\pi
ij)r}\mathrm{d}r\\
& =-\sum_{l=1}^{\infty}\sum_{k_{1}+\cdots+k_{N}=l}\binom{l-1}{k_{1}%
,\ldots,k_{N}}\varepsilon_{1}^{k_{1}}\cdots\varepsilon_{N}^{k_{N}}\int_{a\pm
iK}^{b\pm iK}(\lambda_{1}^{r})^{k_{1}}\cdots(\lambda_{N}^{r})^{k_{N}}\cdot
e^{(-w\pm2\pi ij)r}\mathrm{d}r\text{,}%
\end{align*}
but this integral is easy to compute: We find%
\[
\int_{a\pm iK}^{b\pm iK}(\lambda_{1}^{r})^{k_{1}}\cdots(\lambda_{N}%
^{r})^{k_{N}}\cdot e^{(-w\pm2\pi ij)r}\mathrm{d}r=\left.  \frac{e^{r(k_{1}%
\log\lambda_{1}+\cdots+k_{N}\log\lambda_{N}-w\pm2\pi ij)}}{k_{1}\log
\lambda_{1}+\cdots+k_{N}\log\lambda_{N}-w\pm2\pi ij}\right\vert _{r=a\pm
iK}^{r=b\pm iK}\text{.}%
\]
Now,%
\begin{align*}
& \int_{a\pm iK}^{b\pm iK}\log\left(  1-\sum_{i=1}^{N}\varepsilon_{i}%
\lambda_{i}^{r}\right)  \cdot e^{(-w\pm2\pi ij)r}\mathrm{d}r\\
& =-\sum_{l=1}^{\infty}\sum_{k_{1}+\cdots+k_{N}=l}\binom{l-1}{k_{1}%
,\ldots,k_{N}}\varepsilon_{1}^{k_{1}}\cdots\varepsilon_{N}^{k_{N}}\\
& \qquad\qquad\sum_{j=1}^{\infty}\left.  \frac{e^{r(k_{1}\log\lambda
_{1}+\cdots+k_{N}\log\lambda_{N}-w\pm2\pi ij)}}{k_{1}\log\lambda_{1}%
+\cdots+k_{N}\log\lambda_{N}-w\pm2\pi ij}\right\vert _{r=a\pm iK}^{r=b\pm
iK}\text{.}%
\end{align*}
We still need to settle the uniform convergence: Let $A\subset\mathbf{C}$ be
any compactum avoiding the support of $\mathcal{P}^{\operatorname*{per},\pm}$.
This means that for all $w\in A$ the denominator $k_{1}\log\lambda_{1}%
+\cdots+k_{N}\log\lambda_{N}-w\pm2\pi ij$ is non-zero and thus we can give a
uniform lower bound $>0$ valid on all of $A$. Thus, for our claim it suffices
to handle the numerators. We have the general formula $\left\vert
e^{\alpha\beta}\right\vert =e^{(\operatorname{Re}\alpha)(\operatorname{Re}%
\beta)-(\operatorname{Im}\alpha)(\operatorname{Im}\beta)}$ for all
$\alpha,\beta\in\mathbf{C}$. It yields%
\[
\left\vert e^{(x\pm iy)(k_{1}\log\lambda_{1}+\cdots+k_{N}\log\lambda_{N}%
-w\pm2\pi ij)}\right\vert =(e^{-2\pi y})^{j}\cdot e^{-x\cdot\operatorname{Re}%
w}\cdot e^{\pm y\cdot\operatorname{Im}w}\cdot\prod_{n=1}^{N}e^{k_{n}\left(
x\log\left\vert \lambda_{n}\right\vert \mp y\arg\lambda_{n}\right)  }%
\]
and this can be rewritten as $=(e^{-2\pi y})^{j}\cdot e^{-x\cdot
\operatorname{Re}w}\cdot e^{\pm y\cdot\operatorname{Im}w}\cdot\prod_{n=1}%
^{N}\left(  e^{x\log\left\vert \lambda_{n}\right\vert \mp y\arg\lambda_{n}%
}\right)  ^{k_{n}}$. So far we had worked with $r=x\pm iy$ and $y>0$. We may
instead write $r=x+iy $ and allow all $y\in\mathbf{R}$, $y\neq0$. Then this
expression becomes%
\begin{align}
& =(e^{-2\pi\left\vert y\right\vert })^{j}\cdot e^{-x\cdot\operatorname{Re}%
w}\cdot e^{y\cdot\operatorname{Im}w}\cdot\prod_{n=1}^{N}\left(  e^{x\log
\left\vert \lambda_{n}\right\vert -y\arg\lambda_{n}}\right)  ^{k_{n}%
}\label{lmips1}\\
& =(e^{-2\pi\left\vert \operatorname{Im}r\right\vert })^{j}\cdot
e^{-(\operatorname{Re}r)\cdot(\operatorname{Re}w)}\cdot e^{(\operatorname{Im}%
r)\cdot(\operatorname{Im}w)}\cdot\prod_{n=1}^{N}\left\vert e^{r\log\lambda
_{n}}\right\vert ^{k_{n}}\nonumber
\end{align}
Now, by our choice of $r_{0}$, we have $\left\vert e^{r\log\lambda_{n}%
}\right\vert <\frac{1}{2}$ for all $r\in\lbrack r_{0},+\infty)\times i[-K,K]$
(see Equation \ref{la0}). We observe:\ We have a sum over $k_{1},\ldots,k_{N}$
and $j$. Since $\left\vert \operatorname{Im}r\right\vert >0$ (remember that we
only need the cases $r=a\pm iK$ and $r=b\pm iK$, so we even have $\left\vert
\operatorname{Im}r\right\vert =K$, but this stronger statement is not truly
needed here), the term $(e^{-2\pi\left\vert \operatorname{Im}r\right\vert
})^{j}$ guarantees exponential decay in $j$, and the product term on the right
guarantees exponential decay in each of $k_{1},\ldots,k_{N}$. As a result, the
entire sum over $k_{1},\ldots,k_{N}$ and $j$ can be dominated by convergent
geometric series in each of these variables. Thus, the numerators converge
uniformly in $A$.
\end{proof}

Next, we will use the previous result and let $b\rightarrow+\infty$:

\begin{corollary}
\label{cor_Vplusminuscases}In every compactum $A\subset\mathbf{C}$ inside the
open right half-plane, we have%
\begin{align*}
& \int_{a+iK}^{\infty+iK}\frac{\log\left(  1-\sum_{i=1}^{N}\varepsilon
_{i}\lambda_{i}^{r}\right)  \cdot e^{-wr}}{1-e^{-2\pi ir}}\mathrm{d}r\\
& =-%
{\textstyle\sum_{l=1}^{\infty}}
{\textstyle\sum_{k_{1}+\cdots+k_{N}=l}}
\binom{l-1}{k_{1},\ldots,k_{N}}\varepsilon_{1}^{k_{1}}\cdots\varepsilon
_{N}^{k_{N}}%
{\textstyle\sum_{j=1}^{\infty}}
\frac{e^{(a+iK)(k_{1}\log\lambda_{1}+\cdots+k_{N}\log\lambda_{N}-w+2\pi ij)}%
}{k_{1}\log\lambda_{1}+\cdots+k_{N}\log\lambda_{N}-w+2\pi ij}%
\end{align*}
as a uniformly convergent series in $w\in A$. The series itself converges
uniformly in any compactum in $\mathbf{C}$ which avoids the support of
$\mathcal{P}^{\operatorname*{per},\pm}$. This corresponds to the $V^{+}$-case
of Prop. \ref{prop_Vplusminuscases}; the analogous statements in the $V^{-}%
$-case also holds.
\end{corollary}

\begin{proof}
We continue with the same notation as in the proof of Prop.
\ref{prop_Vplusminuscases}: To show convergence: Let $A\subset\mathbf{C}$ be a
compactum avoiding the support of $\mathcal{P}^{\operatorname*{per},\pm}$. Now
Equation \ref{lmips1} suffices to see uniform convergence. To show agreement
with the integral: This time $A\subset\mathbf{C}$ is a compactum in the
\textit{right half-plane}. Hence, there exists some $\epsilon>0$ such that
$\operatorname{Re}w>\epsilon$ holds for all $w\in A$. Now we take the limit
$b\rightarrow+\infty$. As $\operatorname{Re}w>\epsilon$, the term
$e^{-x\cdot\operatorname{Re}w}\cdot e^{y\cdot\operatorname{Im}w}$ in Equation
\ref{lmips1} therefore gives an exponential decay also in $b\rightarrow
+\infty$; irrespective of the imaginary part as $\left\vert y\right\vert \leq
K$ and $\left\vert \operatorname{Im}w\right\vert $ is also bounded since $w\in
A$, which is compact. Thus, the other factors stay bounded.
\end{proof}

\begin{remark}
If we drop the condition that $A$ has to lie in the right half-plane, the last
claim will indeed fail.
\end{remark}

These proofs have handled two of the integrals which appear in Equation
\ref{l_ModifiedAbelPlana}. For the remaining integrals, the computation can be
carried out in virtually the same way. We leave the details to the reader and
only list the results:

\begin{proposition}
\label{cor_RealAxisTerm}In every compactum $A\subset\mathbf{C}$ inside the
open right half-plane, we have%
\begin{align*}
& \int_{a}^{\infty}\log\left(  1-\sum_{i=1}^{N}\varepsilon_{i}\lambda_{i}%
^{r}\right)  \cdot e^{-wr}\mathrm{d}r\\
& =\sum_{l=1}^{\infty}\sum_{k_{1}+\cdots+k_{N}=l}\binom{l-1}{k_{1}%
,\ldots,k_{N}}\varepsilon_{1}^{k_{1}}\cdots\varepsilon_{N}^{k_{N}}%
\frac{e^{a\left(  k_{1}\log\lambda_{1}+\cdots+k_{N}\log\lambda_{N}-w\right)
}}{k_{1}\log\lambda_{1}+\cdots+k_{N}\log\lambda_{N}-w}\text{.}%
\end{align*}
The series is uniformly convergent in any compactum $A\subset\mathbf{C}$
avoiding the support of $\mathcal{P}$ (i.e. not necessarily contained in the
open right half-plane).
\end{proposition}

\begin{proposition}
For $0<\epsilon<K$, the integrals%
\[
H_{u}^{\pm}(w):=\pm i\int_{\epsilon}^{K}\frac{h(u\pm iy)}{e^{2\pi y}%
-1}\mathrm{d}y
\]
exist and define holomorphic functions in the open right half-plane. Then
$\tilde{H}_{u}^{\pm}(w):=$
\end{proposition}%

\begin{align*}
& \pm e^{-wu}\sum_{l=1}^{\infty}\sum_{k_{1}+\cdots+k_{N}=l}\binom{l-1}%
{k_{1},\ldots,k_{N}}\varepsilon_{1}^{k_{1}}\cdots\varepsilon_{N}^{k_{N}%
}\lambda_{1}^{k_{1}u}\cdots\lambda_{N}^{k_{N}u}\\
& \qquad\qquad i\left.  \frac{e^{(\pm i(k_{1}\log\lambda_{1}+\cdots+k_{N}%
\log\lambda_{N}-w)-2\pi j)y}}{k_{1}\log\lambda_{1}+\cdots+k_{N}\log\lambda
_{N}-w\mp2\pi ij}\right\vert _{y=\epsilon}^{y=K}\text{.}%
\end{align*}

\begin{corollary}
Pick $u\geq r_{0}$ an integer. In every compactum $A\subset\mathbf{C}$ inside
the open right half-plane, we have%
\begin{align*}
& \pm i\int_{0}^{K}\frac{h(u\pm iy)}{e^{2\pi y}-1}\mathrm{d}y=\pm e^{-wu}%
{\textstyle\sum_{l=1}^{\infty}}
{\textstyle\sum_{k_{1}+\cdots+k_{N}=l}}
\binom{l-1}{k_{1},\ldots,k_{N}}\varepsilon_{1}^{k_{1}}\cdots\varepsilon
_{N}^{k_{N}}\lambda_{1}^{k_{1}u}\cdots\lambda_{N}^{k_{N}u}\\
& \qquad\qquad\qquad i\left.  \frac{e^{(\pm i(k_{1}\log\lambda_{1}%
+\cdots+k_{N}\log\lambda_{N}-w)-2\pi j)y}}{k_{1}\log\lambda_{1}+\cdots
+k_{N}\log\lambda_{N}-w\mp2\pi ij}\right\vert _{y=0}^{y=K}\text{.}%
\end{align*}

\end{corollary}

\begin{proof}
[Proof of Theorem \ref{thm_BaseAnalyticContinuation}]We use Equation
\ref{l_ModifiedAbelPlana} and let $b\rightarrow+\infty$. \textit{After} this
limit has been carried out, we invoke Remark \ref{rmk_DecreaseK}: We may run
the same computation for any $K^{\prime}$ such that $0<K^{\prime}<K$ without
having to change $r_{0}$ and our assumptions will remain met. In particular,
we can let $K^{\prime}\rightarrow0$. By inspection, the resulting series then
converge to the statement of the theorem. This finishes the proof.
\end{proof}

\begin{example}
\label{example_AnalyticHelperForEllipticCurveInputData}We make the
continuation provided by Theorem \ref{thm_BaseAnalyticContinuation} explicit:
We pick $q:=2^{2}$. For $N:=3$, let $\lambda_{1}:=q^{-1/2},$ $\lambda
_{2}:=-q^{-1/2}$, $\lambda_{3}:=q^{-1}$, $\varepsilon_{1}:=\varepsilon_{2}:=1$
and $\varepsilon_{3}:=-1$. Below, on the left, we plot the original series
$\sum_{r=1}^{\infty}\log(1-\sum_{i=1}^{N}\varepsilon_{i}\lambda_{i}%
^{r})e^{-wr}$, and on the right we plot the analytic continuation:%
\[
{\includegraphics[
natheight=3.657700in,
natwidth=3.174300in,
height=1.8523in,
width=1.6102in
]%
{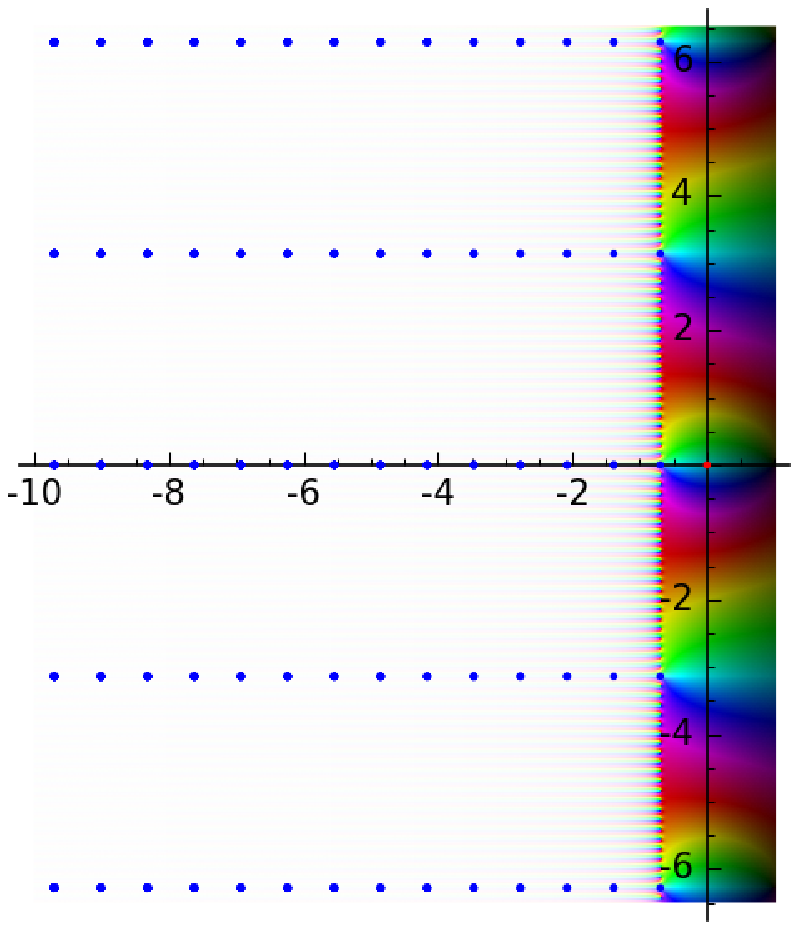}%
}
\qquad\qquad%
{\includegraphics[
natheight=3.657700in,
natwidth=3.174300in,
height=1.8614in,
width=1.6176in
]%
{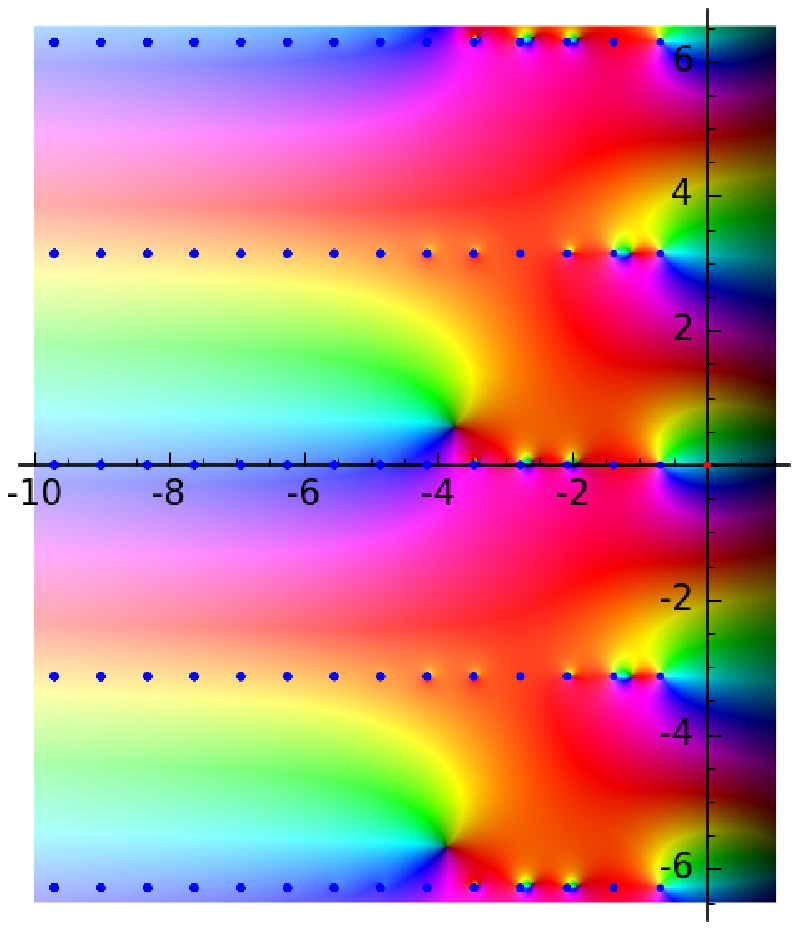}%
}
\]
The dots represent the support of the locally finite pseudo-divisor
$\mathcal{P}^{\operatorname*{per}}$. Secondly, pick $q:=11$. Consider the
polynomial $x^{2}-x+11$. It is irreducible over the rationals. If $\alpha
_{1},\alpha_{2}$ are its two solutions, we have $\left\vert \alpha
_{j}\right\vert =q^{1/2}$ for $j=1,2$, so these are Weil $q$-numbers of weight
one. Pick $N:=3$, $\lambda_{j}:=\alpha_{j}/q$ for $j=1,2$ and $\lambda
_{3}:=1/q$, $\varepsilon_{1}:=\varepsilon_{2}:=1$ and $\varepsilon_{3}:=-1$.%
\[%
{\includegraphics[
natheight=3.657700in,
natwidth=3.174300in,
height=1.8614in,
width=1.6176in
]%
{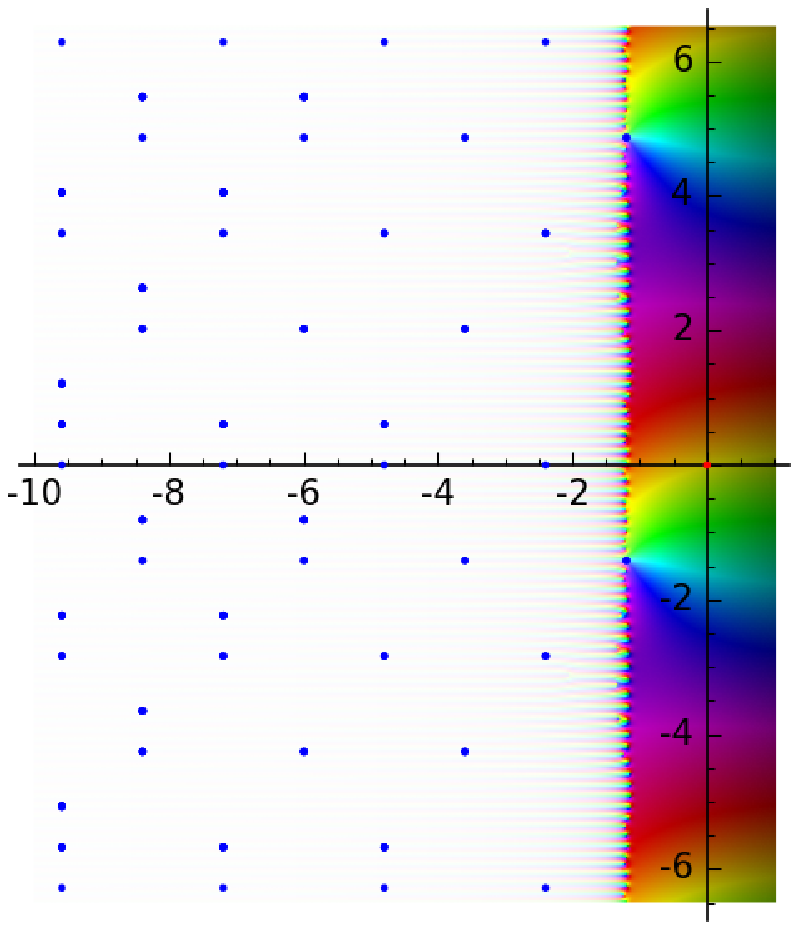}%
}
\qquad\qquad%
{\includegraphics[
natheight=3.657700in,
natwidth=3.174300in,
height=1.8614in,
width=1.6176in
]%
{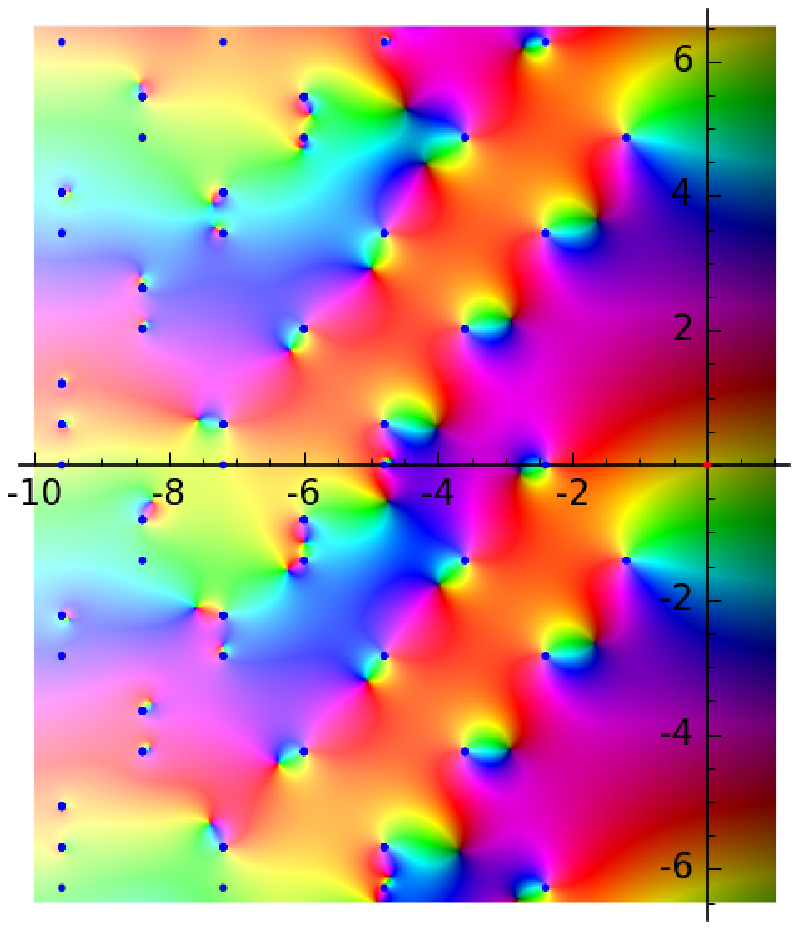}%
}
\]
The input data for this example was not picked at random. See Example
\ref{example_EllipticCurve}.
\end{example}

\begin{example}
The following figure sketches a pseudo-divisor $\mathcal{P}%
^{\operatorname*{per}}$ which fails to be locally finite:%
\[%
{\includegraphics[
height=1.0831in,
width=1.0435in
]%
{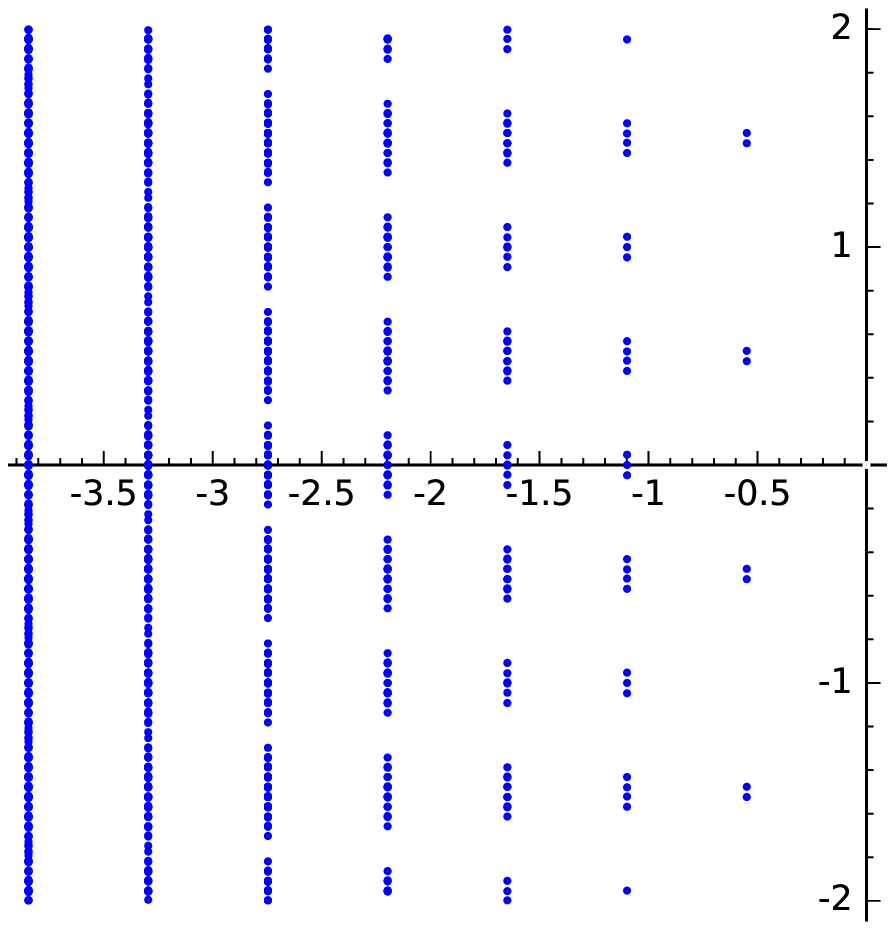}%
}
\]
In such a situation the above picture may represent the locus of poles for a
truncated series for some choice for $b$, but when we take the limit
$b\rightarrow+\infty$ as in the proof of Theorem
\ref{thm_BaseAnalyticContinuation}, these poles accumulate and we cannot hope
for $\left.  \text{\textbf{$\textsc{(AC)}$}}\right.  $ to hold.
\end{example}

\section{Step II}

Suppose we are given $(\underline{\varepsilon},\underline{\lambda
}):=(\varepsilon_{i},\lambda_{i})_{i=1,\ldots,N}$ and have picked $r_{0}$ as
explained in \S \ref{subsect_Setup}. We define a new pseudo-divisor%
\[
\mathcal{D}:=\Phi^{\ast}(\mathcal{P}^{\operatorname*{per}})\qquad
\text{for}\qquad\Phi(w):=e^{-w}\text{,}%
\]
and consider the power series%
\begin{equation}
J_{\underline{\varepsilon},\underline{\lambda},r_{0}}(z):=\sum_{r=r_{0}%
}^{\infty}\log\left(  1-\sum_{i=1}^{N}\varepsilon_{i}\lambda_{i}^{r}\right)
z^{r}\text{.}\label{lza6}%
\end{equation}
The estimate in Remark \ref{rmk_estimate_integrand} implies that this power
series has radius of convergence $\geq1$.

\begin{proposition}
\label{prop_Jtilde}Suppose we are given $(\varepsilon_{i},\lambda
_{i})_{i=1,\ldots,N}$ and have picked $r_{0}$ as explained in
\S \ref{subsect_Setup}. Then for every compactum $A\subset\mathbf{C}^{\times
}\setminus\operatorname*{supp}\mathcal{D}$ and choice of a branch of the
logarithm $\operatorname*{Log}_{A}:A\rightarrow\mathbf{C}$ which extends to a
holomorphic function in some neighbourhood of $A$, the series%
\begin{align*}
\tilde{J}_{\underline{\varepsilon},\underline{\lambda},r_{0}}(z)  & =\frac
{1}{2}\log\left(  1-\sum_{i=1}^{N}\varepsilon_{i}\lambda_{i}^{r_{0}}\right)
z^{r_{0}}\\
& +\sum_{l=1}^{\infty}\sum_{j\in\mathbf{Z}}\sum_{k_{1}+\cdots+k_{N}=l}%
\binom{l-1}{k_{1},\ldots,k_{N}}\varepsilon_{1}^{k_{1}}\cdots\varepsilon
_{N}^{k_{N}}\\
& \qquad\qquad\frac{e^{r_{0}\left(  k_{1}\log\lambda_{1}+\cdots+k_{N}%
\log\lambda_{N}+\operatorname*{Log}_{A}(z)+2\pi ij\right)  }}{k_{1}\log
\lambda_{1}+\cdots+k_{N}\log\lambda_{N}+\operatorname*{Log}_{A}(z)+2\pi
ij}\text{,}%
\end{align*}
is uniformly convergent in $A$. It defines a holomorphic function and is
independent of the choice of $\operatorname*{Log}_{A}$. If $\mathcal{D}$ is
locally finite, $\tilde{J}_{\underline{\varepsilon},\underline{\lambda},r_{0}%
}$ defines a meromorphic function on $\mathbf{C}$ whose poles all have order
one and the support of the divisor of poles agrees with $\operatorname*{supp}%
\mathcal{D}$. Inside the unit circle, $\tilde{J}_{\underline{\varepsilon
},\underline{\lambda},r_{0}}$ agrees with $J_{\underline{\varepsilon
},\underline{\lambda},r_{0}}$.
\end{proposition}

\begin{proof}
Let us write $T(w)$ for the function described by Theorem
\ref{thm_BaseAnalyticContinuation}. Let $z\in\mathbf{C}^{\times}$ be any point
outside the support of $\mathcal{D}$. Let $U$ be an open neighbourhood of $z$
and $\operatorname*{Log}_{U}:U\rightarrow\mathbf{C}$ a branch of the logarithm
which is holomorphic on all of $U$ (always possible after shrinking $U$). We
define a function $g_{U}:U\rightarrow\mathbf{C}$ by the formula $g_{U}%
(u):=T(-\operatorname*{Log}_{U}(u))$ so that, thanks to the definition of
$\mathcal{D}$, $g_{U}$ is holomorphic on $U$. If we do this for opens
$U_{1},U_{2}$ (which overlap), the value of $-\operatorname*{Log}_{U_{j}}$ is
independent of $j=1,2$ up to an integer multiple of $2\pi i$, but since $T$ is
$2\pi i$-periodic in $w$, we will have $g_{U_{1}}\mid_{U_{1}\cap U_{2}%
}=g_{U_{2}}\mid_{U_{1}\cap U_{2}}$. Thus, $g$ glues (without any monodromy!)
and as a result we get a uniquely determined function $g$, defined on all of
$\mathbf{C}^{\times}\setminus\operatorname*{supp}\mathcal{D}$. However, by
Theorem \ref{thm_BaseAnalyticContinuation} this function is locally given by
$\tilde{J}_{\underline{\varepsilon},\underline{\lambda},r_{0}}$ as in the
claim, and this also guarantees the uniform convergence. Since the exponential
function is everywhere locally a homeomorphism, $\mathcal{D}$ is locally
finite if and only if $\mathcal{P}^{\operatorname*{per}}$ is locally finite.
The meromorphy and the statement about the poles follow. We have $\tilde
{J}_{\underline{\varepsilon},\underline{\lambda},r_{0}}=J_{\underline
{\varepsilon},\underline{\lambda},r_{0}}$ inside the open unit disc, as under
$z=e^{-w}$ this corresponds to the open right half-plane.
\end{proof}

\begin{definition}
\label{def_I_U_AsAnIntegral}Let $U\subset\mathbf{C}$ be any simply connected
domain containing $0\in\mathbf{C}$ and not intersecting the support of
$\mathcal{D}$. Define%
\[
I_{\underline{\varepsilon},\underline{\lambda},r_{0}}^{U}(z):=\int_{\gamma
}\frac{\tilde{J}_{\underline{\varepsilon},\underline{\lambda},r_{0}}(w)}%
{w}\mathrm{d}w\text{,}%
\]
where $\gamma$ is any path inside of $U$ from $0\in\mathbf{C}$ to $z\in U$.
\end{definition}

As we demand that $U$ is simply connected, the integral is independent of the
choice of $\gamma$. The integrand $\tilde{J}_{\underline{\varepsilon
},\underline{\lambda},r_{0}}(w)/w$ can be regarded as holomorphic since
$\tilde{J}_{\underline{\varepsilon},\underline{\lambda},r_{0}}$ has a zero of
order $\geq1$ at the origin (by Prop. \ref{prop_Jtilde} it agrees with
$J_{\underline{\varepsilon},\underline{\lambda},r_{0}}$ inside the unit
circle, and by Equation \ref{lza6} and $r_{0}\geq1$, \S \ref{subsect_Setup},
the latter function has no constant coefficient). Thus, $I_{\underline
{\varepsilon},\underline{\lambda},r_{0}}^{U}$ is a holomorphic function on $U$.

\begin{lemma}
\label{lemma_IU}Let $U$ be as in Definition \ref{def_I_U_AsAnIntegral}. At
$z=0$, the function $I_{\underline{\varepsilon},\underline{\lambda},r_{0}}%
^{U}$ has the power series expansion%
\[
I_{\underline{\varepsilon},\underline{\lambda},r_{0}}^{U}(z)=\sum_{r=r_{0}%
}^{\infty}\log\left(  1-\sum_{i=1}^{N}\varepsilon_{i}\lambda_{i}^{r}\right)
\frac{z^{r}}{r}\text{,}%
\]
whose radius of convergence is $\geq1$. In particular, near $z=0$ it is
independent of the choice of $U$.
\end{lemma}

\begin{proof}
By Prop. \ref{prop_Jtilde} in a neighbourhood of the origin, $\tilde
{J}_{\underline{\varepsilon},\underline{\lambda},r_{0}}(w)/w=J_{\underline
{\varepsilon},\underline{\lambda},r_{0}}(w)/w$, so by Equation \ref{lza6} we
have the power series expansion
\[
\frac{J_{\underline{\varepsilon},\underline{\lambda},r_{0}}(w)}{w}%
=\sum_{r=r_{0}}^{\infty}\log\left(  1-\sum_{i=1}^{N}\varepsilon_{i}\lambda
_{i}^{r}\right)  w^{r-1}%
\]
in a neighbourhood of $w=0$ and by termwise integration, we get the power
series in the claim. Termwise integration leaves the radius of convergence invariant.
\end{proof}

Unlike the procedure in the proof of Prop. \ref{prop_Jtilde}, the analytic
continuation $I_{\underline{\varepsilon},\underline{\lambda},r_{0}}^{U}$ will
have non-trivial monodromy.

\begin{proposition}
[Monodromy]\label{prop_Monodromy}We keep the assumptions of the section.
Moreover, suppose the pseudo-divisor $\mathcal{D}$ is locally finite. Let
$\gamma$ be any closed path inside the open set $\mathbf{C}\setminus
\operatorname*{supp}\mathcal{D}$. Then%
\[
\int_{\gamma}\frac{J_{\underline{\varepsilon},\underline{\lambda},r_{0}}%
(w)}{w}\mathrm{d}w\in2\pi i\mathbf{Q}[\varepsilon_{1},\ldots,\varepsilon
_{N}]\text{,}%
\]
where the latter is the subring generated by $\varepsilon_{1},\ldots
,\varepsilon_{N}$ inside $\mathbf{C}$ over the rationals. If $\varepsilon
_{1},\ldots,\varepsilon_{N}$ are algebraic, this is a number field.
\end{proposition}

\begin{proof}
By Proposition \ref{prop_Jtilde} we have%
\begin{align*}
\frac{\tilde{J}_{\underline{\varepsilon},\underline{\lambda},r_{0}}(z)}{z}  &
=\frac{1}{2}\log\left(  1-\sum_{i=1}^{N}\varepsilon_{i}\lambda_{i}^{r_{0}%
}\right)  z^{r_{0}-1}\\
& +\sum_{l=1}^{\infty}\frac{1}{l}\sum_{j\in\mathbf{Z}}\sum_{k_{1}+\cdots
+k_{N}=l}\binom{l}{k_{1},\ldots,k_{N}}\varepsilon_{1}^{k_{1}}\cdots
\varepsilon_{N}^{k_{N}}\\
& \qquad\qquad\frac{e^{r_{0}\left(  k_{1}\log\lambda_{1}+\cdots+k_{N}%
\log\lambda_{N}+\operatorname*{Log}_{A}(z)+2\pi ij\right)  }}{k_{1}\log
\lambda_{1}+\cdots+k_{N}\log\lambda_{N}+\operatorname*{Log}_{A}(z)+2\pi
ij}\cdot\frac{1}{z}\text{,}%
\end{align*}
Note that $r_{0}\geq1$, so the initial term on the right is holomorphic on all
of $\mathbf{C}$. By using the additivity of the integral with respect to the
path and $\mathcal{D}$ being locally finite, it suffices to compute the
integral around sufficiently small circles going around each of the isolated
poles. Concretely, it suffices to compute%
\[
\frac{1}{l}\int_{\eta}\frac{e^{r_{0}(C+\operatorname*{Log}_{A}(z))}%
}{C+\operatorname*{Log}_{A}(z)}\frac{1}{z}\mathrm{d}z
\]
for every choice $C:=k_{1}\log\lambda_{1}+\cdots+k_{N}\log\lambda_{N}+2\pi ij$
and $\eta$ a sufficiently small circle around $z$ such that
$\operatorname*{Log}_{A}(z)=-C$. If we can show that the value of this
integral lies in $2\pi i\mathbf{Q}$, our claim is proven. Write
$C=-\operatorname*{Log}_{A}(W)$ for $W$ suitably chosen (since we only need to
work in some neighbourhood of $z$, this is possible for the same branch). By
the Residue Theorem, since this is a pole of order $1$ at worst,%
\[
=\frac{2\pi i}{l}\underset{z\longrightarrow W}{\lim}(z-W)\cdot\frac
{e^{r_{0}(-\operatorname*{Log}_{A}(W)+\operatorname*{Log}_{A}(z))}%
}{-\operatorname*{Log}_{A}(W)+\operatorname*{Log}_{A}(z)}\frac{1}{z}\text{.}%
\]
Since the derivative of the logarithm, irrespective of the choice of a branch,
is%
\[
\lim_{z\longrightarrow W}\frac{\log z-\log W}{z-W}=\frac{1}{W}\text{,}%
\]
this is easy to compute and we get $=\frac{2\pi i}{l}\cdot W\cdot\frac{1}%
{W}\cdot\underset{z\longrightarrow W}{\lim}e^{r_{0}(-\operatorname*{Log}%
_{A}(W)+\operatorname*{Log}_{A}(z))}=\frac{2\pi i}{l}$. Our claim follows.
\end{proof}

\begin{remark}
As one can see from the proof, as soon as there are poles at all, we will have
monodromy in $2\pi i\mathbf{Q}$, and it will usually not happen that this can
be reduced to $2\pi i\frac{1}{M}\mathbf{Z}$ for some fixed $M\geq1$, even if
$\varepsilon_{1},\ldots,\varepsilon_{N}\in\mathbf{Z}$. To see this, note that
among the coefficients of the series we have $\frac{1}{l}\varepsilon_{j}^{l}$
for all $j=1,\ldots,N$ and all $l\geq1$. As $\varepsilon_{j}$ has a fixed
prime factorization, as $l\rightarrow\infty$, $l$ will infinitely often have
prime factors which cannot be cancelled by $\varepsilon_{j}^{l}$.
\end{remark}

\begin{proposition}
\label{prop_AnalyticCtHelperLevelOne}Suppose we are given $(\varepsilon
_{i},\lambda_{i})_{i=1,\ldots,N}$ and have picked $r_{0}$ as explained in
\S \ref{subsect_Setup}. Suppose the pseudo-divisor $\mathcal{D}$ is locally
finite, and on top of our running assumptions we demand $\varepsilon
_{1},\ldots,\varepsilon_{N}\in\mathbf{Q}$.

\begin{enumerate}
\item Then the power series%
\[
F_{\underline{\varepsilon},\underline{\lambda},r_{0}}(z):=\exp\left(  \frac
{1}{2}\sum_{r=r_{0}}^{\infty}\log\left(  1-\sum_{i=1}^{N}\varepsilon
_{i}\lambda_{i}^{r}\right)  \frac{z^{r}}{r}\right)
\]
has positive radius of convergence.

\item \emph{$\left.  \text{\textbf{$\textsc{(AC)}$}}\right.  $} For every
simply connected domain $U$ with $0\in U\subset\mathbf{C}\setminus
\operatorname*{supp}\mathcal{D}$, there exists a unique holomorphic function
$F_{\underline{\varepsilon},\underline{\lambda},r_{0}}^{U}:U\rightarrow
\mathbf{C}$ such that $F_{\underline{\varepsilon},\underline{\lambda},r_{0}%
}^{U}$ agrees with $F_{\underline{\varepsilon},\underline{\lambda},r_{0}}$ in
some neighbourhood of $t=0$. Moreover, $F_{\underline{\varepsilon}%
,\underline{\lambda},r_{0}}^{U}$ has no zeros in $U$.

\item On the intersection on any two domains $U_{1},U_{2}$ as in (2) we have,%
\[
F_{\underline{\varepsilon},\underline{\lambda},r_{0}}^{U_{1}}/F_{\underline
{\varepsilon},\underline{\lambda},r_{0}}^{U_{2}}\in\mu_{\infty}(U_{1}\cap
U_{2})\text{,}%
\]
i.e. two branches of the analytic continuation differ by a root of unity, and
this fraction is locally constant.

\item The logarithmic derivative $(F_{\underline{\varepsilon},\underline
{\lambda},r_{0}}^{U})^{\prime}/F_{\underline{\varepsilon},\underline{\lambda
},r_{0}}^{U}$ has a single-valued meromorphic continuation to $\mathbf{C}$.
Its locus of poles agrees with $\operatorname*{supp}\mathcal{D}$ and all poles
have order $1$.
\end{enumerate}
\end{proposition}

Note that the logarithmic derivative $(F_{\underline{\varepsilon}%
,\underline{\lambda},r_{0}}^{U})^{\prime}/F_{\underline{\varepsilon
},\underline{\lambda},r_{0}}^{U}$ has a single-valued analytic continuation.

\begin{proof}
By Lemma \ref{lemma_IU} the integral%
\[
I_{\underline{\varepsilon},\underline{\lambda},r_{0}}^{U}(z)=\int_{\gamma
}\frac{\tilde{J}_{\underline{\varepsilon},\underline{\lambda},r_{0}}(w)}%
{w}\mathrm{d}w
\]
defines a holomorphic analytic continuation of $\sum_{r=r_{0}}^{\infty}%
\log\left(  1-\sum_{i=1}^{N}\varepsilon_{i}\lambda_{i}^{r}\right)  \frac
{z^{r}}{r}$ inside the domain $U$. Define $F_{\underline{\varepsilon
},\underline{\lambda},r_{0}}^{U}:=\exp(\frac{1}{2}I_{\underline{\varepsilon
},\underline{\lambda},r_{0}}^{U})$. It follows that $F_{\underline
{\varepsilon},\underline{\lambda},r_{0}}^{U}:U\rightarrow\mathbf{C}$ is
holomorphic, cannot have zeros, and agrees with $F_{\underline{\varepsilon
},\underline{\lambda},r_{0}}$ in a neighbourhood of the origin. This proves
(1) and (2). For (3) and $z\in U_{1}\cap U_{2}$ we compute%
\[
\frac{F_{\underline{\varepsilon},\underline{\lambda},r_{0}}^{U_{1}}%
(z)}{F_{\underline{\varepsilon},\underline{\lambda},r_{0}}^{U_{2}}(z)}%
=\exp\left(  \frac{1}{2}I_{\underline{\varepsilon},\underline{\lambda},r_{0}%
}^{U_{1}}(z)-\frac{1}{2}I_{\underline{\varepsilon},\underline{\lambda},r_{0}%
}^{U_{2}}(z)\right)  =\exp\left(  \frac{1}{2}\left(  \int_{\gamma_{1}}%
-\int_{\gamma_{2}}\right)  \frac{\tilde{J}_{\underline{\varepsilon}%
,\underline{\lambda},r_{0}}(w)}{w}\mathrm{d}w\right)  \text{,}%
\]
where $\gamma_{i}$ is a path inside $U_{i}$ and thus inside $\mathbf{C}%
\setminus\operatorname*{supp}\mathcal{D}$, which goes from $w=0$ to $w=z$.
Hence, $\gamma_{1}\circ\gamma_{2}^{-1}$ is a closed path in $\mathbf{C}%
\setminus\operatorname*{supp}\mathcal{D}$ and by Monodromy (Prop.
\ref{prop_Monodromy}) we get $\exp(\tau)$ for some $\tau\in2\pi i\mathbf{Q}$.
Thus, $\exp(\tau)$ is a root of unity, locally constant, $F_{\underline
{\varepsilon},\underline{\lambda},r_{0}}^{U_{1}}(z)/F_{\underline{\varepsilon
},\underline{\lambda},r_{0}}^{U_{2}}(z)\in\mu_{\infty}$. For (4), note that%
\[
\frac{(F_{\underline{\varepsilon},\underline{\lambda},r_{0}}^{U})^{\prime}%
}{F_{\underline{\varepsilon},\underline{\lambda},r_{0}}^{U}}=\frac{1}{2}%
\frac{\partial I_{\underline{\varepsilon},\underline{\lambda},r_{0}}^{U}%
}{\partial z}(z)=\frac{1}{2}\frac{\tilde{J}_{\underline{\varepsilon
},\underline{\lambda},r_{0}}(z)}{z}%
\]
and we get all the required properties from Proposition \ref{prop_Jtilde} and
the fact that $\tilde{J}_{\underline{\varepsilon},\underline{\lambda},r_{0}}$
has a zero of order $\geq1$ at the origin.
\end{proof}

\begin{theorem}
\label{thm_AnalyticCtHelper}Suppose we are given $(\varepsilon_{i},\lambda
_{i})_{i=1,\ldots,N}$ and have picked $r_{0}$ as explained in
\S \ref{subsect_Setup}. Suppose the pseudo-divisor $\mathcal{D}$ is locally
finite, and on top of our running assumptions we demand $\varepsilon
_{1},\ldots,\varepsilon_{N}\in\mathbf{Q}$. Define a pseudo-divisor%
\[
\mathcal{E}:=\mathcal{D}+c^{\ast}\mathcal{D}\text{,}%
\]
where $c$ denotes complex conjugation on $\mathbf{C}$ and $c^{\ast}$ the
pullback. Then $\mathcal{E}$ is also locally finite.

\begin{enumerate}
\item Then the power series%
\begin{equation}
f(z):=\exp\left(  \sum_{r=r_{0}}^{\infty}\log\left\vert 1-\sum_{i=1}%
^{N}\varepsilon_{i}\lambda_{i}^{r}\right\vert \cdot\frac{z^{r}}{r}\right)
\label{lat1}%
\end{equation}
has positive radius of convergence.

\item \emph{$\left.  \text{\textbf{$\textsc{(AC)}$}}\right.  $} For every
simply connected domain $U$ with $0\in U\subset\mathbf{C}\setminus
\operatorname*{supp}\mathcal{E}$, there exists a unique holomorphic function
$f_{U}:U\rightarrow\mathbf{C}$ such that $f_{U}$ agrees with $f$ in some
neighbourhood of $t=0$. Moreover, $f_{U}$ has no zeros in $U$.

\item On the intersection on any two domains $U_{1},U_{2}$ as in (2) we
locally have%
\[
f_{U_{1}}/f_{U_{2}}\in\mu_{\infty}(U_{1}\cap U_{2})\text{,}%
\]
i.e. two branches of the analytic continuation differ by a root of unity.

\item The logarithmic derivative $f_{U}^{\prime}/f_{U}$ has a single-valued
meromorphic continuation to $\mathbf{C}$. Its locus of poles agrees with
$\operatorname*{supp}\mathcal{E}$ and all poles have order $1$.
\end{enumerate}
\end{theorem}

\begin{proof}
It is easy to see that our assumptions on $(\varepsilon_{i},\lambda_{i})$ in
\S \ref{subsect_Setup} imply that the complex conjugates $(\overline
{\varepsilon_{i}},\overline{\lambda_{i}})$ also satisfy them; perhaps (for a
given fixed $K$) for a different choice of $r_{0}$. However, we can without
loss of generality pick an $r_{0}$ sufficiently large for both $(\varepsilon
_{i},\lambda_{i})$ and $(\overline{\varepsilon_{i}},\overline{\lambda_{i}})$
simultaneously. For $(\overline{\varepsilon_{i}},\overline{\lambda_{i}})$ the
pseudo-divisor $\mathcal{D}$ gets replaced by $c^{\ast}\mathcal{D}$. Next,
note that for all integers $r\geq r_{0}$ we have%
\[
\log\left\vert 1-\sum_{i=1}^{N}\varepsilon_{i}\lambda_{i}^{r}\right\vert
=\frac{1}{2}\log\left(  1-\sum_{i=1}^{N}\varepsilon_{i}\lambda_{i}^{r}\right)
+\frac{1}{2}\log\left(  1-\sum_{i=1}^{N}\overline{\varepsilon_{i}}%
\overline{\lambda_{i}}^{r}\right)
\]
and thus $f=F_{(\varepsilon_{i},\lambda_{i}),r_{0}}\cdot F_{(\overline
{\varepsilon_{i}},\overline{\lambda_{i}}),r_{0}}$. Now apply Prop.
\ref{prop_AnalyticCtHelperLevelOne} to both factors.
\end{proof}

\section{Step III}

So far, we have not looked into the matter of detecting whether the
pseudo-divisor $\mathcal{D}$ (or equivalently $\mathcal{P}%
^{\operatorname*{per}}$) is locally finite.

\begin{lemma}
\label{Lemma_NAtMostTwoImpliesLocalFiniteness}Suppose we are given
$(\varepsilon_{i},\lambda_{i})_{i=1,\ldots,N}$ as explained in
\S \ref{subsect_Setup}. If $N\leq2$, $\mathcal{P}$ is a locally finite
divisor. If $N\leq1$, $\mathcal{P}^{\operatorname*{per},+}$, $\mathcal{P}%
^{\operatorname*{per},-}$ and $\mathcal{P}^{\operatorname*{per}}$ are locally
finite divisors.
\end{lemma}

\begin{proof}
Regarding $\mathcal{P}$, the cases $N=0,1$ are obvious. Suppose $N=2$. Then%
\[
\mathcal{P}=\sum_{k_{1}=0}^{\infty}\sum_{k_{2}=0}^{\infty}\delta_{k_{1}%
+k_{2}\geq1}\frac{(k_{1}+k_{2}-1)!}{k_{1}!k_{2}!}\varepsilon_{1}^{k_{1}%
}\varepsilon_{2}^{k_{2}}\cdot\left[  k_{1}\log\lambda_{1}+k_{2}\log\lambda
_{2}\right]
\]
There are only two cases: (A) If $\log\lambda_{1},\log\lambda_{2}$ are
$\mathbf{R}$-linearly independent, then this spans a two-dimensional cone
(with its apex removed) in the complex plane. In particular, $\mathcal{P}$ is
a locally finite divisor. By our standing assumption $\left\vert \lambda
_{1}\right\vert ,\left\vert \lambda_{2}\right\vert <1$, this cone lies
entirely in the open left half-plane. (B)\ If they are linearly dependent,%
\[
\alpha\log\lambda_{1}+\beta\log\lambda_{2}=0
\]
for some $\alpha,\beta\in\mathbf{R}$, taking the real part shows that
$\alpha,\beta$ must both be non-zero and have opposite parity, say
$\alpha<0<\beta$ without loss of generality. Thus, $\log\lambda_{1}%
=\gamma\cdot\log\lambda_{2}$ for some positive real number $\gamma$. In
particular, for every constant $C^{\prime}>0$ there are only finitely many
$(k_{1},k_{2})$ such that $-C^{\prime}<\operatorname{Re}(k_{1}\log\lambda
_{1}+k_{2}\log\lambda_{2})<0$. In summary, $\mathcal{P}$ lies discretely in a
ray in the negative open half plane. The situation with the perodic divisor
$\mathcal{P}^{\operatorname*{per},\pm}$ is analogous, just replace one
spanning vector by $\pm2\pi i$. Since $\operatorname{Re}\log\lambda_{1}<0$,
this always spans a rank $2$ integral cone, which is discrete.
\end{proof}

\begin{example}
\label{example_AllNAboveThreeCanBeLocallyFinite}Having $N\geq3$ does not
hinder $\mathcal{P}$ or $\mathcal{P}^{\operatorname*{per}}$ to be a locally
finite divisor. We shall construct an example with $N=2^{j}-1$ for any given
$j\geq1$ and $\mathcal{P}^{\operatorname*{per}}$ locally finite: Pick input
data, $\varepsilon_{i}$ and $\lambda_{i}$ with $i\in X:=\{1,2,\ldots,j\}$ as
in \S \ref{subsect_Setup}. Denote by $2^{X}$ the power set and define for all
subsets $I\in2^{X}$ with $\left\vert I\right\vert \geq1$:%
\[
\tilde{\varepsilon}_{I}:=(-1)^{\left\vert I\right\vert -1}\cdot\prod_{i\in
I}\varepsilon_{i}\qquad\text{and}\qquad\tilde{\lambda}_{I}:=\prod_{i\in
I}\lambda_{i}\text{.}%
\]
Clearly we still have $\left\vert \tilde{\lambda}_{I}\right\vert <1$. Hence,
$(\tilde{\varepsilon}_{I},\tilde{\lambda}_{I})$ also determines valid input
data as in \S \ref{subsect_Setup}. Write $\mathcal{\tilde{P}}$ for the
corresponding pseudo-divisor. On the other hand, we can handle for each $i\in
X$, the singleton system $(\varepsilon_{i},\lambda_{i})$, i.e. we remove all
entries except for the one of index $i$, so that seen individually it has
$N=1$. Write $\mathcal{P}_{i}$ for the corresponding pseudo-divisor. Since
$N=1$ for these singleton systems, all $\mathcal{P}_{i}^{\operatorname*{per}}$
are locally finite divisors (Lemma
\ref{Lemma_NAtMostTwoImpliesLocalFiniteness}). We compute%
\[
g(r)=\log\left(  \prod_{i=1}^{N}\left(  1-\varepsilon_{i}\lambda_{i}%
^{r}\right)  \right)  \cdot e^{-wr}=\log\left(  1-\sum_{I\in2^{\{1,\ldots
,N\}},\left\vert I\right\vert \geq1}\tilde{\varepsilon}_{I}\tilde{\lambda}%
_{I}^{r}\right)  \cdot e^{-wr}\text{.}%
\]
Since all $\mathcal{P}_{i}^{\operatorname*{per}}$ are locally finite, Theorem
\ref{thm_BaseAnalyticContinuation} implies that $g$ possesses an analytic
continuation. Thus, the same is true for the function determined by the above
equation. Thus, $\mathcal{\tilde{P}}$ is necessarily also locally finite. To
give more context: If we number the slots of $(\tilde{\varepsilon}_{I}%
,\tilde{\lambda}_{I})$ instead of indexing them by subsets of $2^{X}$, the
pseudo-divisor%
\[
\mathcal{\tilde{P}}:=\sum_{l=1}^{\infty}\sum_{k_{1}+\cdots+k_{2^{j}-1}%
=l}\binom{l-1}{k_{1},\ldots,k_{2^{j}-1}}\tilde{\varepsilon}_{1}^{k_{1}}%
\cdots\tilde{\varepsilon}_{2^{j}-1}^{k_{2^{j}-1}}\cdot\left[  \sum
_{t=1}^{2^{j}-1}k_{t}\log\tilde{\lambda}_{t}\right]
\]
has the property that the integral cone%
\[
\mathbf{Z}_{\geq0}\left\langle k_{1}\log\tilde{\lambda}_{1},\ldots,k_{2^{j}%
-1}\log\tilde{\lambda}_{2^{j}-1}\right\rangle
\]
will (usually, once $j>2$ and a generic choice of $\lambda_{i}$) define a
non-discrete subset of the complex plane, and indeed very generically for
large $j$ be dense. Nonetheless, $\operatorname*{supp}\mathcal{\tilde{P}}$
will always be a locally finite divisor thanks to the (not so obvious) heavy
cancellation of terms, based on the combinatorics of the multinomial coefficients.
\end{example}

Finally, we can prove the existence of an analytic continuation for various
varieties and motives.

\begin{theorem}
\label{Thm_AC_ForAbelianVar}Let $A/\mathbf{F}_{q}$ be an abelian variety of
dimension $g\geq1$. Let $\alpha_{1},\ldots,\alpha_{2g}$ be its Weil numbers.
Then the following properties hold:

\begin{enumerate}
\item \emph{(Radius of Convergence) }The power series $Z_{\log}$ has radius of
convergence precisely one.

\item \emph{$\left.  \text{\textbf{$\textsc{(AC)}$}}\right.  $ }It admits a
holomorphic analytic continuation $\tilde{Z}_{\log}^{U}$ to any simply
connected domain $U\ni0$ avoiding $\{1,\alpha_{1}^{\mathbf{Z}_{\geq2}}%
,\ldots,\alpha_{2g}^{\mathbf{Z}_{\geq2}}\}$.

\item \emph{(Periods) }On the intersection of any two domains $U_{1},U_{2}$ as
in (2), we have $\tilde{Z}_{\log}^{U_{1}}/\tilde{Z}_{\log}^{U_{2}}\in
\mu_{\infty}(U_{1}\cap U_{2})$, i.e. all branches of the analytic continuation
differ by multiplication with a root of unity.

\item \emph{(Logarithmic derivative) }The logarithmic derivative $\tilde
{Z}_{\log}^{U\prime}/\tilde{Z}_{\log}^{U}$ has a single-valued meromorphic
continuation to the entire complex plane with a pole of order $2$ at $z=1$ and
poles of order $1$ at all positive integer powers of the weight one\ Weil
numbers of $A$, i.e. $\{\alpha_{1}^{\mathbf{Z}_{\geq1}},\ldots,\alpha
_{2g}^{\mathbf{Z}_{\geq1}}\}$.

\item \emph{(Monodromy)} Around $z=1$, the function $\tilde{Z}_{\log}$ has an
essential singularity. Around any point $w:=\alpha_{j}^{l}$, $l\geq2$, has an
open neighbourhood in which $\tilde{Z}_{\log}$ has the shape%
\[
(z-w)^{\frac{n_{j}}{l}}\cdot g\text{,}%
\]
where $g$ is non-zero holomorphic in a neighbourhood and $n_{j}$ is the
multiplicity of $\alpha_{j}$ as a root of the Frobenius characteristic polynomial.

\item \emph{(Zeros)} Suppose $A$ is simple. Then the zeros of $\tilde{Z}%
_{\log}^{U}:U\rightarrow\mathbf{C}$, for $U$ as in (2), are precisely%
\[
\{\alpha_{1},\ldots,\alpha_{2g}\}\cap U\text{.}%
\]

\end{enumerate}
\end{theorem}

Statement (6) can also be rephrased as follows: If we consider the analytic
continuation as a lift to the respective covering space where it becomes
single-valued, as in Figure \ref{lpic1}, the zeros of this lift are precisely
the fibers of $\{\alpha_{1},\ldots,\alpha_{2g}\}$ under the covering map.

\begin{proof}
We give a proof without motives: The $\ell$-adic cohomology algebra of $A$ is
canonically isomorphic to an exterior algebra,%
\[
H^{\ast}(A\times_{\mathbf{F}_{q}}\mathbf{F}_{q}^{\operatorname*{sep}%
},\mathbf{Q}_{\ell})=%
{\textstyle\bigwedge\nolimits^{\ast}}
H^{1}(A\times_{\mathbf{F}_{q}}\mathbf{F}_{q}^{\operatorname*{sep}}%
,\mathbf{Q}_{\ell})\text{.}%
\]
As a result, thanks to the Grothendieck--Lefschetz trace formula, we have the
point count%
\[
N_{r}=\left\vert A(\mathbf{F}_{q^{r}})\right\vert =\left\vert \prod_{j=1}%
^{2g}(1-\alpha_{j}^{r})\right\vert \text{,}%
\]
where $\alpha_{1},\ldots,\alpha_{2g}$ are the Weil $q$-numbers (of weight $1$)
of the abelian variety, or equivalently the eigenvalues of the (geometric)
Frobenius, acting on $H^{1}(A\times_{\mathbf{F}_{q}}\mathbf{F}_{q}%
^{\operatorname*{sep}},\mathbf{Q}_{\ell})$ as a Galois module (This particular
result is actually due to Weil and predates the work of the Grothendieck
school). Thus,%
\begin{align*}
\log\left\vert A(\mathbf{F}_{q^{r}})\right\vert  & =\sum_{j=1}^{2g}%
\log\left\vert \alpha_{j}^{r}-1\right\vert =\sum_{j=1}^{2g}\log\left\vert
(\alpha_{j}^{r})\cdot(1-(\alpha_{j}^{-1})^{r})\right\vert \\
& =\sum_{j=1}^{2g}\left(  \frac{r}{2}\log q+\log\left\vert 1-(\alpha_{j}%
^{-1})^{r}\right\vert \right)
\end{align*}
since $\left\vert \alpha_{i}\right\vert =q^{1/2}$ for all $j=1,\ldots,2g$. So
by Definition \ref{def_Zlog_ForVars} the function $Z_{\log}(A,t)$ literally
equals%
\begin{align}
Z_{\log}(A,t)  & =\exp\left(  \sum\nolimits_{r\geq1}\log\left\vert
A(\mathbf{F}_{q^{r}})\right\vert \cdot\frac{t^{r}}{r}\right) \label{lav4}\\
& =\exp\left(  g\log q\frac{t}{1-t}\right)  \cdot\prod_{j=1}^{2g}\exp\left(
\sum\nolimits_{r\geq1}\log\left\vert 1-(\alpha_{j}^{-1})^{r}\right\vert
\cdot\frac{t^{r}}{r}\right)  \text{.}\nonumber
\end{align}
Thus, our claims (2)-(4) are proven if we can show that the desired properties
hold for all factors $\exp\left(  \sum\nolimits_{r\geq1}\log\left\vert
1-(\alpha_{j}^{-1})^{r}\right\vert \cdot\frac{t^{r}}{r}\right)  $. To this
end, we apply Theorem \ref{thm_AnalyticCtHelper} for each $i=1,\ldots,2g$ in
the following situation: $N:=1$, $\varepsilon_{1}:=+1$, $\lambda_{1}%
:=\alpha_{i}^{-1}$ (which has $0<\left\vert \alpha_{i}^{-1}\right\vert <1$ as
required), and $r_{0}=1$. Indeed, $\varepsilon_{1}\in\mathbf{Q}$, and the
resulting pseudo-divisor $\mathcal{P}^{\operatorname*{per}}$ is locally finite
by Lemma \ref{Lemma_NAtMostTwoImpliesLocalFiniteness} since $N=1$. Thus, the
theorem is applicable and we leave it to the reader to compute that the
divisor of poles agrees with $\mathcal{D}$. This settles (1)-(4), albeit only
for the smaller set $\mathbf{C}\setminus\{1,\alpha_{1}^{\mathbf{Z}_{\geq1}%
},\ldots,\alpha_{2g}^{\mathbf{Z}_{\geq1}}\}$. We address (5): Equation
\ref{lav4} settles the essential singularity at $t=1$: The first factor has
such a singularity at $t=1$, while the remaining $2g$ factors can
holomorphically be extended across $t=1$ by Theorem \ref{thm_AnalyticCtHelper}%
. As the logarithmic derivative, by (4), has poles of order $1$ at all points
$w\in\{\alpha_{1}^{\mathbf{Z}_{\geq1}},\ldots,\alpha_{2g}^{\mathbf{Z}_{\geq1}%
}\}$, we locally have%
\[
(\operatorname*{Log}\tilde{Z}_{\log}^{U})^{\prime}(z)=\frac{\beta}%
{z-w}+h(z)\text{,}%
\]
where \textquotedblleft$\operatorname*{Log}$\textquotedblright\ is a locally
defined branch of the logarithm, $\beta\in\mathbf{C}$, and $h$ a holomorphic
function defined in some neighbourhood of $w$. A local computation shows that
that $\beta=\frac{n_{j}}{l}$ for $w=\alpha_{j}^{l}$, where $1\leq j\leq2g$,
$l\geq1$ and $n_{j}\geq1$ is the multiplicity of $\alpha_{j}$ as a root of the
Frobenius characteristic polynomial (as a hint, this local computation is done
as in the proof of Prop. \ref{prop_Monodromy}). Integrating the above equation
then leads to%
\[
(\operatorname*{Log}\tilde{Z}_{\log}^{U})(z)=\frac{n_{j}}{l}\log(z-w)+H(z)
\]
for a new holomorphic function, defined in a neighbourhood. Then $\tilde
{Z}_{\log}^{U}(z)=(z-w)^{n_{j}/l}\cdot\exp(H(z))$, proving (5). Moreover, it
shows that for $l=1$, we have $\frac{n_{j}}{l}\in\mathbf{Z}_{\geq1}$, and thus
the isolated singularities at $\alpha_{1},\ldots,\alpha_{2g}$ are removable,
thus extending (1)-(4) to $\mathbf{C}\setminus\{1,\alpha_{1}^{\mathbf{Z}%
_{\geq2}},\ldots,\alpha_{2g}^{\mathbf{Z}_{\geq2}}\}$. It remains to prove (6):
As we had used Theorem \ref{thm_AnalyticCtHelper}, we know that $\tilde
{Z}_{\log}^{U}$ has no zeros in $U\setminus\{\alpha_{1}^{\mathbf{Z}_{\geq0}%
},\ldots,\alpha_{2g}^{\mathbf{Z}_{\geq0}}\}$, irrespective of what $U$ is.
Thus, it remains to check what happens at the remaining points in
$U$:\ Suppose $A$ is simple, so $n_{j}=1$ for all $j$. At $w=1$, we know that
$\tilde{Z}_{\log}^{U}$ has an essential singularity, and at $w=\alpha_{j}^{l}$
for $i=1,\ldots,2g$ and $l\geq2$, the local description shows that no
holomorphic extension to these points is possible (indeed: The real part
admits a continuous continuation by zero, but the imaginary part has a jump
thanks to the $l$-th root function), so they cannot be contained in the domain
of any analytic continuation. The points $w=\alpha_{j}^{l}$ with $l=1$ and
$w\in U$ remain. As we have seen above, we indeed have zeros at these points.
\end{proof}

\begin{remark}
A motivic proof would use Shermenev's theorem, providing an isomorphism
$\mathcal{M}(A)=%
{\textstyle\bigwedge\nolimits^{\ast}}
h^{1}(A)$. The original paper is \cite{MR0335523} (or as an alternative source
\cite{MR1133323}, \cite{MR1265530}). The proof then proceeds in exactly the
same way.
\end{remark}

By Honda--Tate theory \cite{MR0229642} the abelian varieties over
$\mathbf{F}_{q}$ (up to isogeny) are classified by Galois orbits of Weil
$q$-numbers of weight one:%
\begin{align*}
& \Phi:\{\mathbf{F}_{q}\text{-isogeny classes of abelian varieties
}/\mathbf{F}_{q}\}\\
& \qquad\qquad\overset{\sim}{\longrightarrow}\{\text{Galois orbits of Weil
numbers with }\left\vert x\right\vert ^{2}=q\}
\end{align*}

Thus, knowing $Z_{\log}$ one can explicitly reconstruct the Weil numbers from
the zeros of the analytic continuation, and get the isogeny class of $A$ back.

\begin{corollary}
Given an abelian variety $A/\mathbf{F}_{q}$, the order $1$ poles of the
logarithmic derivative $Z_{\log}$ at points of absolute value $\left\vert
z\right\vert =\sqrt{q}$ are precisely the weight one Weil $q$-numbers of the
abelian variety.
\end{corollary}

The next result covers a wide range of examples. We refer to the Appendix
\S \ref{sect_MotivicPicture} for background and notation regarding motives.

\begin{theorem}
\label{thm_AC_Motive}Suppose $M$ is a motive which decomposes as a finite
direct sum of supersingular motives (e.g. Tate or Artin motives). Suppose it
has a unique top weight (Definition \ref{def_unique_top_weight}). If $Z_{\log
}(M,t)$ is defined (i.e. $N_{l}\geq1$ for all $l\geq1$), then it has
\textbf{$\left.  \text{\textbf{$\textsc{(AC)}$}}\right.  $}. More
specifically: There is a discrete subset $\Delta\subset\mathbf{C}$,
$0\notin\Delta$ such that:

\begin{enumerate}
\item \emph{$\left.  \text{\textbf{$\textsc{(AC)}$}}\right.  $ }For every
simply connected domain $U$ with $0\in U\subset\mathbf{C}\setminus\Delta$,
there exists a unique holomorphic function $\tilde{Z}_{\log}^{U}%
:U\rightarrow\mathbf{C}$ such that $\tilde{Z}_{\log}^{U}$ agrees with
$Z_{\log}$ in some neighbourhood of $t=0$.

\item \emph{(Periods)} On the intersection on any two domains $U_{1},U_{2}$ as
in (1) we have%
\[
\tilde{Z}_{\log}^{U_{1}}/\tilde{Z}_{\log}^{U_{2}}\in\mu_{\infty}(U_{1}\cap
U_{2})\text{,}%
\]
i.e. two branches of the analytic continuation differ by a root of unity.

\item \emph{(Logarithmic derivative)} The logarithmic derivative $(\tilde
{Z}_{\log}^{U})^{\prime}/\tilde{Z}_{\log}^{U}$ has a single-valued meromorphic
continuation to all of $\mathbf{C}$. Its locus of poles agrees with $\Delta$.
With at most finitely many exceptions, the poles have order $1$.

\item \emph{(Monodromy)} In a neighbourhood around any point $w\in\Delta\cap
U$, the continuation $\tilde{Z}_{\log}^{U}$ of (1) locally has the shape%
\[
(z-w)^{\frac{a}{b}}\cdot g\qquad\text{with}\qquad\tfrac{a}{b}\in
\mathbf{Q}\text{,}%
\]
and $g$ some non-zero holomorphic function.
\end{enumerate}
\end{theorem}

\begin{proof}
Suppose $M=\bigoplus_{v}M_{v}$, where each $M_{v}$ consists entirely of
summands of weight $v$. This is possible by assumption. Consider the motivic
(virtual) point count numbers $N_{l}$ as in
\S \ref{subsect_MultZetaForMotivesAndNmCount}. Then $N_{l}=\sum_{v,j}%
(-1)^{v}\alpha_{v,j}^{l}$, where, $\alpha_{v,j}=\zeta_{v,j}q^{v/2}$ for some
root of unity $\zeta_{v,j}$ (depending on $v,j$ as indices; these have nothing
to do with its exponent as a torsion element in the multiplicative group). The
summation over $v$ runs through the individual weights, while $j$ runs through
the collection of eigenvalues appearing in each weight part. By our assumption
of a unique top weight, say it is $q^{m}$ for some $m$, we may write
$N_{r}=\sum_{v<2m,j}(-1)^{v}\alpha_{v,j}^{r}+q^{mr}$. Thus, the definition of
$Z_{\log}(M,t)$ unravels as
\begin{align}
& =\exp\left(  m\log q\frac{t}{1-t}\right)  \exp\left(  \sum_{r=1}^{r_{0}%
-1}\log\left\vert 1+\sum_{v<2m,j}(-1)^{v}\left(  \frac{\alpha_{v,j}}{q^{m}%
}\right)  ^{r}\alpha_{v,j}^{r}\right\vert \cdot\frac{t^{r}}{r}\right)
\nonumber\\
& \qquad\qquad\cdot\exp\left(  \sum_{r=r_{0}}^{\infty}\log\left\vert
1+\sum_{v<2m,j}(-1)^{v}\left(  \frac{\alpha_{v,j}}{q^{m}}\right)  ^{r}%
\alpha_{v,j}^{r}\right\vert \cdot\frac{t^{r}}{r}\right)  \text{,}\label{l4a}%
\end{align}
where $\left\vert \alpha_{v,j}/q^{m}\right\vert =q^{(v-2m)/2}$ and since
$v<2m$, we have $0<\left\vert \alpha_{v,j}/q^{m}\right\vert <1$. Next, we wish
to apply Theorem \ref{thm_AnalyticCtHelper} with the datum $(\varepsilon
_{i},\lambda_{i})$ given by $\varepsilon_{i}\in\{1,-1\}$ and $\lambda_{i}$
running through the values $(\alpha_{v,j}/q^{m})$ for all summands which
appear; and pick $r_{0}$ sufficient for the assumptions of the theorem to be
applicable (this is possible: the above computation works for any $r_{0}\geq
1$, and by our remarks in \S \ref{subsect_Setup} any sufficiently large choice
of $r_{0}$ meets the conditions). To this end, it only remains to check that
the pseudo-divisor $\mathcal{D}$ is locally finite. This is equivalent to
checking that $\mathcal{P}^{\operatorname*{per}}$ is locally finite (as
$\mathcal{D}$ is just the pullback of the latter under a map which is
everywhere a local homeomorphism). However,%
\[
\mathcal{P}^{\operatorname*{per}}=\sum_{j=-\infty}^{\infty}\sum_{l=1}^{\infty
}\sum_{k_{1}+\cdots+k_{N}=l}\binom{l-1}{k_{1},\ldots,k_{N}}\varepsilon
_{1}^{k_{1}}\cdots\varepsilon_{N}^{k_{N}}\cdot\left[  k_{1}\log\lambda
_{1}+\cdots+k_{N}\log\lambda_{N}+2\pi ij\right]  \text{,}%
\]
where all $\lambda_{s}$ are of the shape $(\zeta_{v(s),j(s)}q^{v(s)/2})/q^{m}%
$, so our $\mathcal{P}^{\operatorname*{per}}$ satisfies%
\[
\operatorname*{supp}\mathcal{P}^{\operatorname*{per}}\subseteq\overline
{\bigcup_{\substack{k_{1},\ldots,k_{N}\geq0 \\j\in\mathbf{Z}}}\{k_{1}%
\log\lambda_{1}+\cdots+k_{N}\log\lambda_{N}+2\pi ij\}}%
\]
where, splitting these points into their real and imaginary part,%
\[
2\pi ij+\sum_{s=1}^{N}k_{s}\log\lambda_{s}=\sum_{s=1}^{N}k_{s}\left(
\frac{v(s)-2m}{2}\right)  \log(q)+i\left(  2\pi j+\sum_{s=1}^{N}k_{s}%
\arg(\zeta_{v(s),j(s)})\right)  \text{.}%
\]
Since the $\zeta_{v(s),j(s)}$ are all roots of unity and the sum is finite,
there exists some fixed integer $M\geq1$ such that we have $\arg
(\zeta_{v(s),j(s)})\in\frac{1}{M}2\pi i\mathbf{Z}$. Hence, all these values
are contained in the lattice spanned by $\mathbf{Z}\left\langle \frac{1}%
{2}\log q,\frac{1}{M}2\pi i\right\rangle $. Hence, $\operatorname*{supp}%
\mathcal{P}^{\operatorname*{per}}$ is contained in a discrete rank $2$ lattice
in the complex plane, and thus necessarily locally finite. Hence, we can
indeed invoke the Theorem and it guarantees that the last factor, line
\ref{l4a}, has $\left.  \text{\textbf{$\textsc{(AC)}$}}\right.  $ for
$\mathbf{C}\setminus\operatorname*{supp}\mathcal{D}$. The remaining two
factors are%
\[
\exp\left(  m\log(q)\cdot\frac{t}{1-t}\right)  \exp\left(  \sum_{r=1}%
^{r_{0}-1}\log\left\vert 1+\sum_{v<2m,j}(-1)^{v}\left(  \frac{\alpha_{v,j}%
}{q^{m}}\right)  ^{r}\alpha_{v,j}^{r}\right\vert \cdot\frac{t^{r}}{r}\right)
\text{.}%
\]
In either case, we face the exponential of a function which is rational in $t
$. Clearly, this is immediately defined on all of $\mathbf{C}$ except for the
isolated pole locus of the rational function in question. This proves $\left.
\text{\textbf{$\textsc{(AC)}$}}\right.  $ for $Z_{\log}(M,t)$ for some set
$\Delta:=\operatorname*{supp}\mathcal{D}\cup\{$finite set of points$\}$, so if
$\mathcal{D}$ is locally finite, $\Delta$ is necessarily a discrete subset of
the complex plane. Note that any multi-valuedness, i.e. distinct branches, can
only stem from the factor controlled by Theorem \ref{thm_AnalyticCtHelper},
i.e. line \ref{l4a}, so the description of the possible monodromy remains
valid for $Z_{\log}$. It remains to prove (4): As the logarithmic derivative
is meromorphic on the entire complex plane, locally around $w\in\Delta$ we
have%
\[
(\operatorname*{Log}\tilde{Z}_{\log}^{U})^{\prime}(z)=\frac{\beta}{(z-w)^{m}%
}+h(z)\text{,}%
\]
where \textquotedblleft$\operatorname*{Log}$\textquotedblright\ is a locally
defined branch of the logarithm, $m\in\mathbf{Z}_{\geq1}$, $\beta\in
\mathbf{C}$ some constant, and $h$ a meromorphic function defined in some
neighbourhood of $w$ of pole order at most $m-1$. Thus, if $m\geq2$,%
\[
(\operatorname*{Log}\tilde{Z}_{\log}^{U})(z)=\frac{1}{1-m}\frac{\beta
}{(z-w)^{m-1}}+H(z)\text{,}%
\]
for suitable $H$, meromorphic of pole order at most $m-2$, and so $\tilde
{Z}_{\log}^{U}$ has an essential singularity nearby $w$. If $m=1$,%
\[
(\operatorname*{Log}\tilde{Z}_{\log}^{U})(z)=\beta\log(z-w)+H(z)
\]
for suitable $H$, and by Monodromy (Prop. \ref{prop_Monodromy}) the monodromy
of $(\operatorname*{Log}\tilde{Z}_{\log}^{U})^{\prime}$ is rational, so
$\beta\in\mathbf{Q}$. Thus, $\tilde{Z}_{\log}^{U}=(z-w)^{\beta}\exp(H(z))$ in
an open neighbourhood of $w$.
\end{proof}

\begin{theorem}
\label{Thm_AC_ForTateAndSupersingMotivicDecomp}Suppose $X/\mathbf{F}_{q}$ is a
geometrically connected smooth projective variety

\begin{enumerate}
\item whose motive decomposes as a finite direct sum of supersingular motives
(e.g. Artin or Tate motives), and

\item which has a $\mathbf{F}_{q}$-rational point.
\end{enumerate}

Then $Z_{\log}(X,t)$ is defined, and both have \textbf{$\left.
\text{\textbf{$\textsc{(AC)}$}}\right.  $}.
\end{theorem}

\begin{proof}
Since $X$ is geometrically connected, Poincar\'{e} Duality implies that it has
a unique top weight. The function $Z_{\log}$ is defined since we have a
$\mathbf{F}_{q}$-rational point, so $N_{l}\geq1$ for all $l\geq1$. Thus, we
can apply Theorem \ref{thm_AC_Motive} to the motive of the variety and are done.
\end{proof}

\subsection{Cellularity and the functions $\Lambda_{n}$%
\label{subsect_Cellularity}}

\begin{corollary}
\label{cor_cellular}Suppose $X/\mathbf{F}_{q}$ is a geometrically connected
smooth projective cellular variety. Then $Z_{\log}(X,t)$ has \textbf{$\left.
\text{\textbf{$\textsc{(AC)}$}}\right.  $}.
\end{corollary}

\begin{proof}
If the variety is cellular, even its Chow motive (i.e. $\mathsf{Mot}_{\sim
}(k)$ with rational equivalence) splits into a finite direct sum of Tate
motives, see \cite[Theorem 7.2]{MR2110630} or \cite[Corollary 6.11]%
{MR1758562}. This induces the same statement for $\mathsf{Mot}_{hom_{\ell}%
}(k)$ and $\mathsf{Mot}_{num}(k)$. Now use the previous theorem.
\end{proof}

The following definition turns out to be convenient:

\begin{definition}
Define%
\[
\Lambda_{n}(t):=Z_{\log}(\mathbf{A}^{n}-\{0\},t)\text{,}%
\]
i.e.%
\[
\Lambda_{n}(t)=\exp\left(  \sum\nolimits_{r\geq1}\log(q^{nr}-1)\cdot
\frac{t^{r}}{r}\right)  \text{.}%
\]

\end{definition}

Theorem \ref{thm_AnalyticCtHelper} easily implies that \textbf{$\left.
\text{\textbf{$\textsc{(AC)}$}}\right.  $ }holds\ (we have $N=1$, so by\ Lemma
\ref{Lemma_NAtMostTwoImpliesLocalFiniteness} the finiteness of $\mathcal{P}%
^{\operatorname*{per}}$ is automatic).

\begin{example}
\label{example_ProjectiveSpace}For projective space we have%
\[
\left\vert \mathbf{P}^{n}(\mathbf{F}_{q^{r}})\right\vert =1+q^{r}%
+q^{2r}+\cdots+q^{nr}\text{.}%
\]
As $\mathbf{P}^{n}$ is a cellular variety, we could directly invoke Cor.
\ref{cor_cellular}. However, we will handle this example manually. We compute%
\begin{align*}
Z_{\log}(\mathbf{P}^{n},t)  & =\exp\left(  \sum_{r\geq1}\log\left(  \sum
_{m=0}^{n}(q^{r})^{m}\right)  \cdot\frac{t^{r}}{r}\right) \\
& =\frac{\exp\left(  \sum_{r\geq1}\log\left(  (q^{r})^{n+1}-1\right)
\cdot\frac{t^{r}}{r}\right)  }{\exp\left(  \sum_{r\geq1}\log\left(
q^{r}-1\right)  \cdot\frac{t^{r}}{r}\right)  }\\
& =\frac{\Lambda_{n+1}(t)}{\Lambda_{1}(t)}\text{.}%
\end{align*}
Hence, besides \textbf{$\left.  \text{\textbf{$\textsc{(AC)}$}}\right.  $ }we
have a suggestive result: At least over an algebraically closed base field we
can also interpret $\mathbf{P}^{n}$ as $(\mathbf{A}^{n+1}-\{0\})/\sim$, where
the equivalence relation identifies all points on a shared line, i.e. on a
shared $\mathbf{A}^{1}$.
\end{example}

\begin{example}
For the general linear group one finds%
\begin{equation}
\left\vert \operatorname{GL}_{k}(\mathbf{F}_{q})\right\vert =q^{\frac
{k(k-1)}{2}}(q^{k}-1)(q^{k-1}-1)\cdots(q-1)\text{.}\label{la100}%
\end{equation}
We can directly plug this into the definition of $Z_{\log}$. Thanks to the
multiplicative nature of this formula, this leads to a factorization of the
function: Namely, we compute%
\begin{align*}
Z_{\log}(\operatorname{GL}_{k},t)  & =\exp\left(  \sum_{r\geq1}\log\left\vert
(q^{r})^{\frac{k(k-1)}{2}}((q^{r})^{k}-1)((q^{r})^{k-1}-1)\cdots
(q^{r}-1)\right\vert \cdot\frac{t^{r}}{r}\right) \\
& =\exp\left(  \frac{k(k-1)}{2}\log q\cdot\frac{t}{1-t}\right)  \cdot
\exp\left(  \sum_{r\geq1}\log\left(  \prod_{l=1}^{k}(q^{rl}-1)\right)
\cdot\frac{t^{r}}{r}\right)  \text{.}%
\end{align*}
Pulling the product out of the logarithm, we get the factorization%
\[
=q^{\frac{k(k-1)}{2}\cdot\frac{t}{1-t}}\cdot\prod_{l=1}^{k}\Lambda
_{l}(t)\text{.}%
\]
Although the linear group $\operatorname{GL}_{k}$ is not a projective variety,
it follows that \textbf{$\left.  \text{\textbf{$\textsc{(AC)}$}}\right.  $ }holds.
\end{example}

\begin{example}
In a similar fashion, one can treat the Grassmannians,%
\begin{align*}
\left\vert G(k,n)(\mathbf{F}_{q})\right\vert  & =\frac{(q^{n}-1)(q^{n-1}%
-1)\cdots(q^{n-k+1}-1)}{(q^{k}-1)\cdots(q-1)}\\
& =\left.  \prod_{l=n-k+1}^{n}(q^{l}-1)\right/  \prod_{l=1}^{k}(q^{l}%
-1)\text{.}%
\end{align*}
Being cellular, \textbf{$\left.  \text{\textbf{$\textsc{(AC)}$}}\right.  $
}itself immediately follows from Corollary \ref{cor_cellular}. However, much
as in the previous example, we also get a pleasant formula:%
\begin{align*}
Z_{\log}(G(k,n),t)  & =\frac{\exp\left(  \sum_{r\geq1}\log\left(
\prod_{l=n-k+1}^{n}(q^{rl}-1)\right)  \cdot\frac{t^{r}}{r}\right)  }%
{\exp\left(  \sum_{r\geq1}\log\left(  \prod_{l=1}^{k}(q^{rl}-1)\right)
\cdot\frac{t^{r}}{r}\right)  }\\
& =\left.  \prod_{l=n-k+1}^{n}\Lambda_{l}(t)\right/  \prod_{l=1}^{k}%
\Lambda_{l}(t)\text{.}%
\end{align*}
The properties of the analytic continuation follow from the properties of the
$\Lambda_{l}(t)$ alone.
\end{example}

\begin{example}
Suppose $Q$ is a quadratic form and consider the inhomogeneous affine solution
space $X:=\operatorname*{Spec}\mathbf{F}_{q}[t_{1},\ldots,t_{n}]/(Q-\alpha)$
for some non-zero $\alpha\in\mathbf{F}_{q}^{\times}$. While being cellular,
$X$ is not projective. We follow \cite[Theorem 2.7]{MR1680530}: If $Q$ is of
Type 1, we have (for suitable $m$)%
\[
\left\vert X(\mathbf{F}_{q^{r}})\right\vert =q^{r(n-m)}\left(  q^{rm}%
-q^{r(m-1)/2}\right)  =\left(  q^{n-\frac{m-1}{2}}\right)  ^{r}\cdot\left(
q^{\left(  \frac{m-1}{2}\right)  r}-1\right)  \text{.}%
\]
Thus,%
\[
Z_{\log}(X,t)=\left(  q^{\left(  n-\frac{m-1}{2}\right)  }\right)  ^{\frac
{t}{1-t}}\cdot\Lambda_{\frac{m-1}{2}}(t)\text{.}%
\]
Again, the \textbf{$\left.  \text{\textbf{$\textsc{(AC)}$}}\right.  $ }falls
out from the properties of these two factors although $X$ itself is not projective.
\end{example}

If $X$ is smooth projective cellular, the functions $\Lambda_{n}$ obviously
are a convenient ingredient to decompose $Z_{\log}$ into factors. If one
allows non-projective cellular varieties, an additional term of the shape
$q^{(\ldots)\frac{t}{1-t}}$ plays a r\^{o}le. Some of the above computations
have a deeper structural reason on the level of motives. This following all
directly follows from the work of N. Karpenko \cite{MR1758562}, who in turn
attributes the basic idea (in the case of quadrics) to M. Rost.

\begin{proposition}
\label{prop_ExactForGrass}Let $V$ be a finite-dimensional $\mathbf{F}_{q}%
$-vector space. Let $\operatorname*{Grass}(V)$ denotes the full Grassmannian
of all vector subspaces of $V$ (of any dimension). Then $\operatorname*{Grass}%
(V)$ is a smooth projective $\mathbf{F}_{q}$-variety. For every short exact
sequence%
\[
0\longrightarrow V^{\prime}\longrightarrow V\longrightarrow V^{\prime\prime
}\longrightarrow0
\]
of finite-dimensional vector spaces, one has%
\[
Z_{\log}(\operatorname*{Grass}V)=Z_{\log}(\operatorname*{Grass}V^{\prime
})\cdot Z_{\log}(\operatorname*{Grass}V^{\prime\prime})\text{.}%
\]
The analogous statement holds for the varieties of length $r$ flags, for any
$r$.
\end{proposition}

This follows immediately from \cite[Corollary 9.13]{MR1758562} resp.
\cite[Corollary 11.5]{MR1758562}: The Grassmannian of $V$ is not the product
of the Grassmannians of $V^{\prime}$ and $V^{\prime\prime}$, but as it just
differs by a fibration into affine spaces, the motive is the tensor product
motive nonetheless.

\subsection{Curves}

\begin{theorem}
\label{thm_caseofcurves}Suppose $X/\mathbf{F}_{q}$ is a geometrically
connected smooth projective curve with an $\mathbf{F}_{q}$-rational point. If

\begin{enumerate}
\item the genus is $g=0,1$ or

\item the genus is $g\geq2$ and the Jacobian of $X$ is supersingular,
\end{enumerate}

then $Z_{\log}(X,t)$ has $\left.  \text{\textbf{$\textsc{(AC)}$}}\right.  $.
\end{theorem}

The supersingular condition can be checked if one understands global $1$-forms:

\begin{theorem}
[{Nygaard \cite[Theorem 4.1]{MR654203}}]\label{thm_Nygaard}Suppose
$X/\mathbf{F}_{q}$ is a geometrically connected smooth projective curve with
an $\mathbf{F}_{q}$-rational point. Suppose the Cartier operator
$C:H^{0}(X,\Omega^{1})\longrightarrow H^{0}(X,\Omega^{1})$ induces the zero
map. Then the Jacobian of $X$ is supersingular.
\end{theorem}

Most people appear to expect that there exist curves of arbitrarily high genus
and supersingular Jacobian over any finite field, so that this theorem would
give a rich supply of high genus curves with $\left.
\text{\textbf{$\textsc{(AC)}$}}\right.  $. However, it is not easy to make
this claim solid:

\begin{problem}
[{van der Geer \cite[Problem 19]{MR1812812}}]Do there exist smooth projective
curves of arbitrary genus with supersingular Jacobian over all finite fields?
\end{problem}

To the best of our knowledge this problem is only settled (and affirmatively
so) in the case of characteristic two \cite{MR1310953}.

We have the feeling that the converse of (2) might have a chance to be true.

\begin{problem}
Is it true: A geometrically connected smooth projective curve with an
$\mathbf{F}_{q}$-rational curve and genus $g\geq2$ has $\left.
\text{\textbf{$\textsc{(AC)}$}}\right.  $ if and only if the Jacobian is supersingular?
\end{problem}

\begin{proof}
[Proof of Theorem \ref{thm_caseofcurves}]If $X$ has genus $0$ and a rational
point, it must be $\mathbf{P}^{1}$. Thus, it is cellular and the claim follows
from Corollary \ref{cor_cellular}. If $X$ has genus $1$ and a rational point,
it is an elliptic curve and Theorem \ref{Thm_AC_ForAbelianVar} applies.
Finally, suppose $X$ has arbitrary genus and $\operatorname*{Pic}%
\nolimits^{0}(X)$ is supersingular. As $X$ has an $\mathbf{F}_{q}$-rational
point by assumption, $Z_{\log}$ is defined (i.e. we have $N_{r}\geq1$ for all
$r\geq1$). Moreover, the $\ell$-adic homological motive of the curve splits as
$\mathcal{M}(X)=\mathbf{Z}\oplus h^{1}(\operatorname*{Pic}\nolimits^{0}%
(X))\oplus\mathbf{Z}(1)$. As the Jacobian is supersingular, its motive, and in
particular its weight one part $h^{1}(\operatorname*{Pic}\nolimits^{0}(X))$
(the entire motive is the full exterior algebra over this weight one part) is
supersingular. Hence, $\mathcal{M}(X)$ splits as a finite direct sum of Tate
motives and supersingular motives, so Theorem
\ref{Thm_AC_ForTateAndSupersingMotivicDecomp} applies.
\end{proof}

\begin{example}
\label{example_EllipticCurve}When we invoke Theorem \ref{Thm_AC_ForAbelianVar}
for a supersingular elliptic curve over $\mathbf{F}_{2^{2}}$, or an ordinary
elliptic curve over $\mathbf{F}_{11}$ with Frobenius characteristic polynomial
$x^{2}-x+11$, the input data for our constructions as in
\S \ref{subsect_Setup}, corresponds to those in Example
\ref{example_AnalyticHelperForEllipticCurveInputData}.
\end{example}

\subsection{Linear recurrences}

\begin{theorem}
\label{thm_LinRecur}If $X/\mathbf{F}_{q}$ is a smooth projective variety of
dimension $\geq1$, meeting the hypotheses of Theorem \ref{thm_i_1} or Theorem
\ref{thm_i_2}, then the sequence%
\[
n\mapsto\log\left\vert X(\mathbf{F}_{q^{n}})\right\vert
\]
does not satisfy any linear recurrence equation.
\end{theorem}

\begin{proof}
If the coefficients of a power series satisfy a linear recurrence, then the
power series describes a rational function, and thus it has a meromorphic
analytic continuation to the entire complex plane. Thus, in our situation,
this continuation agrees with the ones given for the logarithmic derivative of
$Z_{\log}$ by Theorem \ref{Thm_AC_ForAbelianVar} resp. \ref{thm_AC_Motive}. In
either case, the locus of poles is governed by a suitable pseudo-divisor.
However, a rational function has only finitely many poles, so we reach a
contradiction as soon as we can show that the relevant pseudo-divisors have
support larger than a finite set of points. In either case this is easy to see.
\end{proof}

\appendix

\section{\label{sect_MotivicPicture}Construction for motives}

In this appendix we collect a survey on motives and discuss how to extend the
definition of $Z_{\log}$ to motives. In many ways this appears to be the more
natural habitat for the theory.

Recall our conventions from \S \ref{subsect_Conventions}.

\subsection{Numerical motives over a finite field}

The category $(\mathsf{Mot}_{num}(\mathbf{F}_{q}),\otimes_{twisted})$ is an
abelian semi-simple $F$-linear Tannakian category. Let us briefly recall the
ingredients for this: (1) Thanks to Jannsen's Theorem \cite[Theorem
1]{MR1150598} any category of numerical motives $\mathsf{Mot}_{num}(k)$ over
an arbitrary field $k$ is $F$-linear abelian semi-simple. (2) Numerical
motives have a canonical finite weight decomposition,%
\[
M=\bigoplus_{i}h^{i}(M)\text{,}%
\]
where the sum runs over finitely many $i$, depending on $M$. We call
$h^{i}(M)$ the \emph{weight }$i$\emph{\ part}. This is based on the
algebraicity of K\"{u}nneth projectors, following Katz--Messing \cite[\S III]%
{MR0332791}. (3) The na\"{\i}ve tensor product on $\mathsf{Mot}_{num}(k)$ can
impossibly yield a Tannakian category. However, using a twisted tensor product
due to Deligne one can resolve this issue over finite fields \cite[Corollary
2]{MR1150598} and Remark (2) following this Corollary, loc. cit.

If $M$ is a numerical motive (so, concretely $=(X,p)$ for $X$ a smooth
projective variety and $p$ a correspondence, idempotent up to numerical
equivalence), then the Frobenius of $X$ gives a well-defined endomorphism,
usually denoted by $\pi_{X}$, of $X$. Its characteristic polynomial has
coefficients in $\mathbf{Q}$. Define the ordinary zeta function by%
\[
Z^{num}(M,t):=\prod_{r}\det\left(  1-\pi_{X}\cdot t\left\vert h^{r}(M)\right.
\right)  ^{(-1)^{r+1}}%
\]
where $h^{r}(X)$ is the weight $r$ part. See \cite[Prop. 2.1]{MR1265538} for
details. As the individual characteristic factors are polynomials, and there
are only finitely many non-zero weight parts, $Z^{num}(M,t)$ is a rational
function. We also observe $Z^{num}(M,0)=1$ for all $M$. For more on the
Tannakian viewpoint, see \cite{MR2520468}.

\subsection{Homological motives over a finite field}

On the other hand, $\mathsf{Mot}_{hom_{\ell}}(\mathbf{F}_{q})$ denotes
homological motives with respect to $\ell$-adic cohomology, $\ell
\neq\operatorname*{char}k$. That is, define%
\begin{equation}
X\longmapsto H_{\ell}(X):=\bigoplus H^{i}(X\times_{\mathbf{F}_{q}}%
\mathbf{F}_{q}^{\operatorname*{sep}},\mathbf{Q}_{\ell})\label{la45}%
\end{equation}
for smooth projective varieties $X$. This is functorial in homological
correspondences. If $(X,p)$ is a homological motive and $\pi_{X}$ again
denotes the Frobenius, one may define the ordinary zeta function by%
\begin{align*}
& Z^{hom_{\ell}}((X,p),t):=\prod_{r}\det\left(  1-\pi_{X\ast}\cdot t\left\vert
H^{r}((X,p),\mathbf{Q}_{\ell})\right.  \right)  ^{(-1)^{r+1}}\\
& \qquad\qquad\qquad\qquad\qquad\qquad=\prod_{r}\det\left(  1-\pi_{X\ast}\cdot
t\left\vert p_{\ast}H^{r}(X,\mathbf{Q}_{\ell})\right.  \right)  ^{(-1)^{r+1}%
}\text{,}%
\end{align*}
where $p_{\ast}H^{r}(X,\mathbf{Q}_{\ell})$ denotes the direct summand of the
$\ell$-adic cohomology which is cut out by the idempotent $p$, and $\pi
_{X\ast}$ denotes the action of the Frobenius, as induced to cohomology.
Again, we observe $Z^{hom_{\ell}}(M,0)=1$ for all $M$.

\begin{Fact}
Both constructions yield the same zeta function, $Z^{hom_{\ell}}=Z^{num}$, for
all $\ell\neq\operatorname*{char}k$.
\end{Fact}

\subsection{Tate motives, supersingular motives}

A simple numerical motive has a characteristic polynomial via the Tannakian
structure, as in \cite[Prop. 2.1]{MR1265538}. A simple $\ell$-adic homological
motive has a characteristic polynomial by taking the action of the\ Frobenius
on its $\ell$-adic cohomology. Thus, either way, we have a notion of
characteristic polynomial and we will call its roots the \emph{Frobenius
eigenvalues}. We shall use the following conventions:

\begin{definition}
We call a simple (numerical or homological) motive \emph{Tate} if its
Frobenius eigenvalues are of the shape $q^{w}$ for some $w\in\mathbf{Z}$. We
call it \emph{supersingular} if its Frobenius eigenvalues are of the shape%
\[
\zeta\cdot q^{w/2}%
\]
for some $w\in\mathbf{Z}$ and $\zeta$ any root of unity.
\end{definition}

If we believe in the Tate conjecture, the motives over $\mathbf{F}_{q}$ are
(Tannakian) generated by abelian varieties and the supersingular abelian
varieties then generate the same motives as when we take the supersingular
ones with the above definition. This justifies the term `supersingular'.

\begin{example}
We have $\mathcal{M}(\mathbf{P}^{n})=\mathbf{Z}\oplus\mathbf{Z}(1)\oplus
\cdots\oplus\mathbf{Z}(n)$, sitting in $h^{0}$,$h^{2}$,\ldots, $h^{2n}$
respectively. In particular, these summands are all supersingular.
\end{example}

\subsection{Multiplicative zeta for
motives\label{subsect_MultZetaForMotivesAndNmCount}}

The following considerations make sense both in $\ell$-adic homological or
numerical motives; and under sending a homological motive to its numerical
counterpart, they are compatible.

So, let $Z$ denote either $Z^{hom_{\ell}}$ or $Z^{num}$, according to which
viewpoint we may prefer. Expanding it as a power series around $t=0$, we may
define numbers $N_{m}$ for a motive $(X,p)$ by%
\[
\sum_{m\geq1}N_{m}\frac{t^{m}}{m}:=\log Z((X,p),t)\text{.}%
\]
To make sense of this, we note that $Z((X,p),t)$ at $t=0$ is $+1$. Hence,%
\[
\log Z((X,p),t)=\log\left(  1-\left(  \sum_{r\geq1}a_{r}t^{r}\right)  \right)
=-\sum_{l\geq1}\frac{1}{l}\left(  \sum_{r\geq1}a_{r}t^{r}\right)  ^{l}%
\]
makes sense as a (formal or genuine) power series, and moreover its constant
coefficient vanishes. Thus, we may write it in the form $\log Z((X,p),t)=\sum
\nolimits_{r\geq1}N_{r}\cdot\frac{t^{r}}{r}$ for uniquely determined values
$N_{r}\in\mathbf{Q}$. Now, define%
\begin{equation}
Z_{\log}(X,t):=\exp\left(  \sum_{r\geq1}\log\left\vert N_{r}\right\vert
\cdot\frac{t^{r}}{r}\right)  \text{.}\label{lzims1}%
\end{equation}
This may not make sense for arbitrary motives since we could (and can!) have
$N_{r}=0$. This corresponds to the issue that we can only define $Z_{\log}$ in
the context of varieties when we demand the existence of a rational point. We
will content ourselves with this definition, which will make sense for many
motives. Otherwise, we will simply say that the multiplicative zeta function
is not defined. Nonetheless, a convenient observation is the following:

\begin{definition}
\label{def_unique_top_weight}We say that a motive $M$ has \emph{unique top
weight} if it decomposes as a finite direct sum%
\[
M=M_{1}\oplus\cdots\oplus M_{r}\oplus\mathbf{Z}(m)
\]
such that the\ Frobenius eigenvalues of all factors $M_{i}$, $i=1,\ldots,r$,
are strictly $<q^{m}$.
\end{definition}

\begin{lemma}
If $M$ has unique top weight, $N_{r}$ can only vanish for finitely many $r$
(and this can be effectively bounded).
\end{lemma}

This condition is clearly met for geometrically connected smooth projective
varieties $X/\mathbf{F}_{q}$ because the Poincar\'{e} dual partner of the
$h^{0}(X)=\mathbf{Z}(0)$ of the single connected component provides the single
summand $\mathbf{Z}(d)$ with $d:=\dim X$.

\begin{proof}
The ordinary zeta function of a motive has the shape%
\[
Z(M,t)=\prod\nolimits_{i,j}(1-\alpha_{i,j}\cdot t)^{(-1)^{i+1}}\text{,}%
\]
where $\left\vert \alpha_{i,j}\right\vert =q^{i/2}$ are\ Weil $q$-numbers of
weight $i$ (this follows from working with $\ell$-adic homological motives,
using the Weil conjectures there, and then using equality of characteristic
polynomials for numerical vs. $\ell$-adic homological motives). Thus,%
\begin{align*}
\sum_{l\geq1}N_{l}\frac{t^{l}}{l}  & =\sum_{i,j}(-1)^{i+1}\log(1-\alpha
_{i,j}\cdot t)\\
& =\sum_{i,j}(-1)^{i}\sum_{l=1}^{\infty}\frac{1}{l}\alpha_{i,j}^{l}t^{l}%
=\sum_{l=1}^{\infty}\left(  \sum_{i,j}(-1)^{i}\alpha_{i,j}^{l}\right)
\frac{t^{l}}{l}\text{.}%
\end{align*}
Thus, $N_{l}=\sum_{i,j}(-1)^{i}\alpha_{i,j}^{l}$. By our assumption precisely
one $\alpha_{i,j}$ has an absolute value strictly larger than any other of the
$\alpha_{i,j}$. Call the corresponding index $(i_{top},j_{top}) $. Then%
\[
N_{l}=(-1)^{i_{top}}\alpha_{i_{top,j_{top}}}^{l}(1+\sum_{(i,j)\neq
(i_{top},j_{top})}(-1)^{i-i_{top}}(\alpha_{i,j}/\alpha_{i_{top},j_{top}})^{l})
\]
with $\left\vert \alpha_{i,j}/\alpha_{i_{top},j_{top}}\right\vert <1$ for all
$(i,j)\neq(i_{top},j_{top})$. For sufficiently large $l$, the term in the
bracket is too close to $1$ to ever vanish again.
\end{proof}

We give some examples which show phenomena specific to $Z_{\log}$ for general
pure motives:

\begin{example}
\label{example_SupersingEllCurve}Suppose $A/\mathbf{F}_{p}$ is a supersingular
elliptic curve. Then for $h^{1}(A)$ we get $N_{r}=-\left(  \left(  \sqrt
{p}\right)  ^{r}+\left(  -\sqrt{p}\right)  ^{r}\right)  =-(1+(-1)^{r})p^{r/2}%
$. Every $N_{2r+1}$ is zero. Hence, we cannot define $Z_{\log}(h^{1}(A),t)$.
\end{example}

This kind of problem is specific to motives. For smooth projective varieties
the assumption of having an $\mathbf{F}_{q}$-rational point settles $N_{r}%
\neq0$ for all $r\geq1$.

\begin{example}
We continue Example \ref{example_SupersingEllCurve}. Define $Y:=h^{0}(A)\oplus
h^{1}(A)$, i.e. we truncate the top degree summand from the motive. Although
the simple summands of the motive $Y$ are all supersingular, Theorem
\ref{thm_AC_Motive} does \textsl{not} apply because $Y $ does not have a
unique top weight (Definition \ref{def_unique_top_weight}). However, we can
still establish \textbf{$\left.  \text{\textbf{$\textsc{(AC)}$}}\right.  $
}manually: We get%
\[
N_{r}=1-(1+(-1)^{r})p^{r/2}%
\]
and thus, after some series manipulations,%
\begin{align}
& =\exp\left(  -\frac{\log2}{2}\log\left(  1-t^{2}\right)  \right)  \cdot
\exp\left(  \frac{1}{2}\log(p)\frac{t^{2}}{1-t^{2}}\right) \nonumber\\
& \qquad\cdot\exp\left(  \frac{1}{2}\sum\nolimits_{r\geq1}\log\left\vert
1-\frac{1}{2}(p^{-1})^{r}\right\vert \cdot\frac{(t^{2})^{r}}{r}\right)
\text{.}\label{lze1}%
\end{align}
For $N:=1$, $\varepsilon_{1}:=2$, $\lambda_{1}:=p^{-1}$, the assumptions of
Theorem \ref{thm_AnalyticCtHelper} are met (and the respective $\mathcal{D}$
is locally finite thanks to $N=1$ and\ Lemma
\ref{Lemma_NAtMostTwoImpliesLocalFiniteness}) and we get a multi-valued
analytic continuation $W$ in the sense of Definition \ref{def_AC} such that in
a neighbourhood of $t=0$,%
\[
Z_{\log}(Y,t)^{2}=\left(  1-t^{2}\right)  ^{-\log(2)}\cdot p^{\left(
\frac{t^{2}}{1-t^{2}}\right)  }\cdot W(t^{2})\text{.}%
\]
The square on the left accounts for $\frac{1}{2}$ in Equation \ref{lze1},
while $t^{2}$ on the right accounts for the squared variable in loc. cit.
Since the existence of an analytic continuation for $W$ implies the existence
of a continuation for $t\mapsto\sqrt{W(t^{2})}$, we get \textbf{$\left.
\text{\textbf{$\textsc{(AC)}$}}\right.  $ }for $Z_{\log}(Y,t)$. This is our
first example where we needed $\varepsilon_{1}\neq\pm1$.
\end{example}

\begin{example}
We continue Example \ref{example_SupersingEllCurve} in a different way. Define
$W:=h^{1}(A)\oplus h^{2}(A)$, i.e. this time we truncate the degree zero part.
Since $h^{2}(A)\cong\mathbf{Z}(2)$, $W$ has supersingular summands and unique
top weight. Theorem \ref{thm_AC_Motive} applies. We compute%
\[
N_{r}=-(1+(-1)^{r})p^{r/2}+p^{r}%
\]
and after some series manipulations, this leads to
\[
Z_{\log}(W,t)=p^{\left(  \frac{t}{1-t}\right)  }\exp\left(  \frac{1}{2}%
\sum\nolimits_{r\geq1}\log\left\vert 1-2(p^{-1})^{r}\right\vert \cdot
\frac{(t^{2})^{r}}{r}\right)  \text{.}%
\]
With $N:=1,$ $\varepsilon_{1}:=2$ and $\lambda_{1}:=p^{-1}$, the Theorem
\ref{thm_AnalyticCtHelper} can be applied directly.
\end{example}

\bibliographystyle{amsalpha}
\bibliography{ollinewbib}
Date: {\today}%

\end{document}